\documentclass[reqno]{amsproc}[15p]
\usepackage{bbm}
\usepackage{amsfonts}
\usepackage{mathrsfs,amscd,amssymb,amsthm,amsmath,bm,graphicx,psfrag,subfigure,xypic}
\input epsf
\newtheorem{theorem}{Theorem}[section]
\newtheorem{lemma}[theorem]{Lemma}

\newtheorem{thm}[theorem]{Theorem}
\newtheorem{prop}[theorem]{Proposition}
\newtheorem{cor}[theorem]{Corollary}

\theoremstyle{definition}

\theoremstyle{remark}

\numberwithin{equation}{section}

\usepackage{bm}

\def\mod{{\rm mod\;}}
\def\Mod{{\rm Mod}}

\def\Rep{{\rm Rep}}

\def\supp{{\rm supp\,}}

\def\im{{\rm im\,}}

\def\TN{\textsc{tn}}
\def\FL{\textsc{fl}}
\def\AC{\textsc{ac}}
\def\TC{\textsc{tc}}
\def\CU{\textsc{c}}
\def\EU{\textsc{e}}
\def\CE{\textsc{ce}}

\def\CTF{\textsc{ctf}}

\def\sp{\hspace{2ex}}

\begin{document}

\title[Dual Complementary polynomials and Tutte polynomial]
{Dual complementary polynomials of graphs and a
combinatorial-geometric interpretation on the values of the Tutte
polynomial at positive integers}

\author{Beifang Chen}
\address{Department of Mathematics,
Hong Kong University of Science and Technology,
Clear Water Bay, Kowloon, Hong Kong}
\curraddr{}

\email{mabfchen@ust.hk}
\thanks{Research is supported by RGC Competitive Earmarked Research Grants 600608, 600409, and 600811.}


\date{28 October 2005}



\begin{abstract}
We introduce modular (integral) complementary polynomial $\kappa$
($\kappa_{\Bbb Z}$) of two variables on a graph $G$ by counting the
number of modular (integral) complementary tension-flows. We further
introduce cut-Eulerian equivalence relation on orientations and
geometric structures: complementary open lattice polyhedron
$\Delta_\CTF$, 0-1 polytope $\Delta^+_\CTF$, and lattice polytopes
$\Delta^\rho_\CTF$ with respect to orientations $\rho$. The
polynomial $\kappa$ ($\kappa_{\Bbb Z}$) is a common generalization
of the modular (integral) tension polynomial $\tau$ ($\tau_{\Bbb
Z}$) and the modular (integral) flow polynomial $\varphi$
($\varphi_\Bbb{Z}$) of one variable, and can be decomposed into a
sum of product Ehrhart polynomials of complementary open 0-1
polytopes. There are dual complementary polynomials $\bar\kappa$ and
$\bar\kappa_{\Bbb Z}$, dual to $\kappa$ and $\kappa_{\Bbb Z}$
respectively, in the sense that the lattice-point counting to the
Ehrhart polynomials is taken inside a topological sum of the dilated
closed polytopes $\bar\Delta^+_\CTF$. It turns out remarkably that
$\bar\kappa$ is Whitney's rank generating polynomial $R_G$, which
gives rise to a nontrivial combinatorial-geometric interpretation on
the values of the Tutte polynomial $T_G$ at all positive integers.
In particular, some special values of $\kappa_{\Bbb Z}$ and
$\bar\kappa_{\Bbb Z}$ ($\kappa$ and $\bar\kappa$) count the number
of certain special kinds (of equivalence classes) of orientations,
including the recovery of a few well-known values of $T_G$.
\end{abstract}

\maketitle

\section{Introduction}

The Tutte polynomial $T_G(x,y)$ of a graph $G$ is a common
generalization of the chromatic polynomial $\chi(G,t)$ and the flow
polynomial $\varphi(G,t)$, and is one of the most important
polynomials in graph theory. Unlike definitions of $\chi$ by
counting proper colorings and of $\varphi$ by counting nowhere-zero
flows, $T_G$ is defined by Whitney's rank generating polynomial
$R_G(x,y)$, rather than by counting certain combinatorial objects;
see \cite[p.339]{Bollobas1} and \cite[p.45]{Welsh1}. It has been
wondered for a long time if there exists a counting style definition
for $T_G$. In fact, the combinatorial meanings of $T_G$ at a few
special values, such as $T_G(i,j)$ with $1\leq i,j\leq 2$, can be
read out directly from $R_G$; see Theorem 5 in
\cite[p.345]{Bollobas1}. However, finding combinatorial/geometric
interpretations on the values of $T_G$ at integers has been
continually an active research since Tutte \cite{Tutte}. The
classical interpretations of $T_G$ at a family of special integers
were made by Tutte (see, for example, \cite{Chen-I,Chen-II}) as
follows:
\begin{align}
\tau(G,t)&=(-1)^{r(G)}T_G(1-t,0), \label{Tau-Tutte}\\
\varphi(G,t)&=(-1)^{n(G)}T_G(0,1-t), \label{Phi-Tutte}
\end{align}
where $\tau(G,t)$ ($=\chi(G,t)/t^{c(G)}$) is the tension polynomial
of $G$ and $c(G)$ is the number of connected components. Several
other combinatorial interpretations were made from various
viewpoints: Crapo and Rota's finite field interpretation of
$|T_M(1-q^k,0)|$ on a matroid $M$ \cite{Crapo-Rota1}; Stanley's
interpretation of $|\chi(G,-t)|=t^{c(G)}|T_G(1+t,0)|$ with $t\geq 1$
\cite{Stanley1} and its dual version on $|\varphi(G,-1)|$ by Green
and Zaslavsky \cite{Greene-Zaslvsky1}; Greene's interpretation as
the weight enumerator of linear codes \cite{Greene} and its
generalization by Barg \cite{Barg} and by Green and Zaslavsky
\cite{Greene-Zaslvsky1}; Jaeger's interpretation of linear code and
dual code words \cite{Jaeger}; Brylawksi and Oxley's two-variable
coloring formula \cite{Brylawski-Oxley}, etc.

More recently, Kook, Reiner, and Stanton \cite{Kook-Reiner-Stanton}
found a convolution formula on the Tutte polynomial of a matroid
$M$:
\begin{equation}\label{Conv-Formula-Tutte-M}
T_M(x,y)=\sum_{X\subseteq M} T_{M/X}(x,0)\,T_{M|X}(0,y),
\end{equation}
which was used by Reiner \cite{Reiner1} to give interpretations of
$T_M$ and typically of $T_G$ at nonpositive integers or other
co-related numbers. Kochol \cite{Kochol1,Kochol2} introduced
integral tension polynomial $\tau_{\Bbb Z}(G,t)$ and integral flow
polynomial $\varphi_{\Bbb Z}(G,t)$, which are closely related to
$\tau$ and $\varphi$, and these polynomials led him to define
integral and modular tension-flow polynomials in \cite{Kochol3}.
Gioan \cite{Gioan} gave combinatorial interpretations of $T_G$ at
the special integers $(i,j)$ with $1\leq i,j\leq 2$, using
cycle-cocycle reversing systems. And the very recent work of Chang,
Ma, and Yeh \cite{Chang-Ma-Yeh1} on a new expression of $T_G$, using
graph parking functions.

In the present paper, we study systematically the complementary
tension-flows (CTF) of a graph $G$ and introduce dual complementary
polynomials. Fix an orientation $\varepsilon$ (see Section~2) on $G$
to have a digraph $(G,\varepsilon)$ throughout. Each function
$h\in{\Bbb R}^E$ is decomposed automatically and uniquely into
$h=f+g$, where $f$ is a tension and $g$ is a flow of
$(G,\varepsilon)$; the ordered pair $(f,g)$ is known as a {\em
tension-flow} of $(G,\varepsilon)$. We consider those functions
$h=f+g\in{\Bbb R}^E$ whose tension-flows $(f,g)$ satisfy the
so-called {\em complementary condition:}
\begin{equation}
f(e)g(e)=0,\quad f(e)+g(e)\neq 0 \quad \mbox{for all $e\in E$.}
\end{equation}
A tension-flow $(f,g)$ is said to be a {\em $(p,q)$-tension-flow},
where $p,q$ are positive integers, if $|f(e)|<p$ and $|g(e)|<q$ for
all $e\in E$. We denote by $K(G,\varepsilon)$ the space of all
complementary tension-flows of $(G,\varepsilon)$, and by
$K(G,\varepsilon;p,q)$ the space of all complementary
$(p,q)$-tension-flows. The complementary condition is motivated by
the work of Reiner \cite{Reiner1}.

Let $\Delta_\CTF(G,\varepsilon)$ denote the relatively open lattice
polyhedron of all complementary $(1,1)$-tension-flows of
$(G,\varepsilon)$, and $\Delta^+_\CTF(G,\varepsilon)$ the relatively
open 0-1 polytope of all nonnegative complementary
$(1,1)$-tension-flows. For each orientation $\rho$ on $G$, let
$\Delta^\rho_\CTF(G,\varepsilon)$ denote the relatively open lattice
polytope of complementary $(1,1)$-tension-flows $(f,g)$ of
$(G,\varepsilon)$ such that $f(e)+g(e)>0$ if
$\rho(v,e)=\varepsilon(v,e)$ and $f(e)+g(e)<0$ if
$\rho(v,e)\neq\varepsilon(v,e)$ for each edge $e$ at its one
endvertex $v$. We shall see that $\Delta_\CTF(G,\varepsilon)$ is a
disjoint union of $\Delta^\rho_\CTF(G,\varepsilon)$, where $\rho$ is
extended over all orientations on $G$. Each polytope
$\Delta^\rho_\CTF(G,\varepsilon)$ is lattice isomorphic to the 0-1
polytope $\Delta^+_\CTF(G,\rho)$, and can be decomposed into a
product
\begin{equation}
\Delta^+_\CTF(G,\rho) = \Delta_\TN^+(G,B_\rho) \times
\Delta_\FL^+(G,C_\rho),
\end{equation}
where $C_\rho$ is the maximal strong subdigraph of $(G,\rho)$,
$B_\rho$ is the subdigraph induced by the edge set $E-E(C_\rho)$,
$\Delta_\TN^+(G,B_\rho)$ is the relatively open 0-1 polytope
consisting of 1-tensions $f$ of $(G,\rho)$ such that $f|_{B_\rho}>0$
and $f|_{C_\rho}=0$, and $\Delta_\FL^+(G,C_\rho)$ is the relatively
open 0-1 polytope consisting of 1-flows $g$ of $(G,\rho)$ such that
$g|_{B_\rho}=0$ and $g|_{C_\rho}>0$.

Let ${\mathcal O}(G)$ denote the set of all orientations on $G$. The
key ingredient of the paper is to view the closure
$\bar\Delta^\rho_\CTF(G,\varepsilon)$ as a dual of
$\Delta^\rho_\CTF(G,\varepsilon)$, and to view the topological sum
\begin{equation}
\widetilde{\Delta}_\CTF(G,\varepsilon):=\sum_{\rho\in{\mathcal
O}(G)} \bar\Delta^\rho_\CTF(G,\varepsilon)\; \mbox{(union of
disjoint copies)}
\end{equation}
as a dual to the non-convex polyhedron $\Delta_\CTF(G,\varepsilon)$.
We apply the Ehrhart theory to the above lattice polyhedron and
lattice polytopes.

For positive integers $p,q$, let $(p,q)\Delta_\CTF(G,\varepsilon)$
denote the dilation of $\Delta_\CTF(G,\varepsilon)$ in two
independent parameters, consisting of tension-flows $(pf,qg)$ with
$(f,g)\in\Delta_\CTF(G,\varepsilon)$. Then
$K(G,\varepsilon;p,q)=(p,q)\Delta_\CTF(G,\varepsilon)$. We define
the polynomial counting functions
\begin{align}
\kappa_{\Bbb Z}(G;p,q) & =\big|(p,q)\Delta_\CTF(G,\varepsilon) \cap
({\Bbb Z}\times{\Bbb Z})^E\big|, \label{kappa-Z-defn}\\
\kappa_\varepsilon(G;p,q) & =\big|(p,q)\Delta^+_\CTF(G,\varepsilon)
\cap ({\Bbb Z}\times{\Bbb Z})^E\big|.
\end{align}

For nonnegative integers $p,q$, the dilation
$(p,q)\bar\Delta^\rho_\CTF(G,\varepsilon)$ consists of tension-flows
$(pf,qg)$ with $(f,g)\in\bar\Delta^\rho_\CTF(G,\varepsilon)$. Then
$(p,q)\widetilde\Delta_\CTF(G,\varepsilon)$ is a topological sum of
$(p,q)\bar\Delta^\rho_\CTF(G,\varepsilon)$ extended over
$\rho\in{\mathcal O}(G)$. We define the dual polynomial counting
functions
\begin{align}
\bar\kappa_{\Bbb Z}(G;p,q) & =
\big|(p,q)\widetilde\Delta_\CTF(G,\varepsilon)
\cap ({\Bbb Z}\times{\Bbb Z})^E\big|,\\
\bar\kappa_\varepsilon(G;p,q) & =
\big|(p,q)\bar\Delta^+_\CTF(G,\varepsilon) \cap ({\Bbb Z}\times{\Bbb
Z})^E\big|.
\end{align}
Then $\kappa_{\Bbb Z}(G;,p,q)$ ($\kappa_\varepsilon(G;p,q)$) counts
the number of (nonnegative) integer-valued complementary
$(p,q)$-tension-flows of $(G,\varepsilon)$;
$\bar\kappa_\varepsilon(G;p,q)$ counts the number of integer-valued
tension-flows $(f,g)$ of $(G,\varepsilon)$ such that $0\leq f(e)\leq
p$ and $0\leq g(e)\leq q$ for all $e\in E$; and
\begin{equation}
\bar\kappa_{\Bbb Z}(G;p,q) =\sum_{\rho\in\mathcal{O}(G)}
\bar\kappa_\rho(G;p,q).
\end{equation}
We call $\kappa_{\Bbb Z}$ ($\kappa_\varepsilon$) the {\em integral
(local) complementary polynomial} of $G$ {\em with respect to
$\varepsilon$}, and $\bar\kappa_{\Bbb Z}$ ($\bar\kappa_\varepsilon$)
the {\em dual integral (local) complementary polynomial}.

There is a unimodular isomorphism between $\Delta^+_\CTF(G,\rho)$
and $\Delta^+_\CTF(G,\sigma)$, whenever $\rho,\sigma$ differ exactly
on an edge-disjoint union of a locally directed cut and a directed
Eulerian subgraph, said to be {\em cut-Eulerian equivalent}, denoted
$\rho\sim_\CE\sigma$. Indeed, the cut-Eulerian equivalence is an
equivalence relation on $\mathcal{O}(G)$. Moreover,
\[
\kappa_\rho(G;x,y)=\kappa_\sigma(G;x,y) \quad \mbox{if} \quad
\rho\sim_\CE\sigma.
\]
Let $[\mathcal{O}(G)]$ denote the set of cut-Eulerian equivalence
classes $[\rho]$, where $\rho\in\mathcal{O}(G)$. We introduce the
polynomial counting function
\begin{align}
\bar\kappa(G;p,q):&=
\sum_{[\rho]\in[\mathcal{O}(G)]}\bar\kappa_\rho(G;p,q).
\end{align}
It turns out remarkably that $\bar\kappa$ is the same as the rank
generating
polynomial $R_G$. 

Let $A,B$ be abelian groups of orders $|A|=p$, $|B|=q$. We define
the polynomial counting function
\begin{equation}\label{kappa-defn}
\kappa(G,\varepsilon;p,q)=|K(G,\varepsilon;A,B)|,
\end{equation}
where $K(G,\varepsilon;A,B)$ is the set of ordered pairs $(f,g)$
such that $f$ is an $A$-tension and $g$ is a $B$-flow of
$(G,\varepsilon)$, and $f(e)=0$ if and only if $g(e)\neq 0$ for all
$e\in E$. We call $\kappa$ ($\bar\kappa$) the {\em (dual) modular
complementary polynomial} of $G$. We summarize our main results as
the following theorems.

\begin{thm}\label{Integral-Thm}
{\rm (a)} The counting function $\kappa_{\Bbb Z}(G;p,q)$ {\em
($\bar\kappa_{\Bbb Z}(G;p,q)$)} is a polynomial function of positive
(nonnegative) integers $p,q$, having the same degree as the Tutte
polynomial $T_G$, and is independent of the chosen orientation
$\varepsilon$.

{\rm (b)} {\em Decomposition Formulas:}
\begin{align}
\kappa_{\Bbb Z}(G;x,y) &= \sum_{\rho\in\mathcal{O}(G)}
\kappa_{\rho}(G;x,y), \label{Int-Kappa-Decom}\\
\bar\kappa_{\Bbb Z}(G;x,y) &= \sum_{\rho\in\mathcal{O}(G)}
\bar\kappa_{\rho}(G;x,y). \label{Int-Bar-Kappa-Decom}
\end{align}

{\rm (c)} {\em Reciprocity Laws:}
\begin{align}
\kappa_{\Bbb Z}(G;-x,-y) &= \sum_{\rho\in\mathcal{O}(G)}
(-1)^{r(G)+|E(C_\rho)|} \bar\kappa_\rho(G;x,y),\label{Int-Kappa-Reciprocity}\\
\bar\kappa_{\Bbb Z}(G;-x,-y) &= \sum_{\rho\in\mathcal{O}(G)}
(-1)^{r(G)+|E(C_\rho)|} \kappa_\rho(G;x,y).
\label{Int-Bar-Kappa-Reciprocity}
\end{align}

{\rm (d)} {\em Specializations:}
\begin{equation}\label{Int-Kappa-Special-Variable}
\kappa_{\Bbb Z}(G;x,1)=\tau_{\Bbb Z}(G,x),\sp \kappa_{\Bbb
Z}(G;1,y)=\varphi_{\Bbb Z}(G,y),
\end{equation}
\begin{equation}\label{Int-Bar-Kappa-Special-Variable}
\bar\kappa_{\Bbb Z}(G;x,-1)=\bar\tau_{\Bbb Z}(G,x), \sp
\bar\kappa_{\Bbb Z}(G;-1,y)=\bar\varphi_{\Bbb Z}(G,y).
\end{equation}

{\rm (e)} {\em Convolution Formulas:}
\begin{equation}
\kappa_{\Bbb Z}(G;x,y) =\sum_{X\subseteq E} \tau_{\Bbb Z}(G/X,x)\,
\varphi_{\Bbb Z}(G|X,y), \label{Int-Kappa-Conv-Formula}
\end{equation}
\begin{equation}
\bar\kappa_{\Bbb Z}(G;x,y) =\sum_{X\subseteq E} \bar\tau_{\Bbb
Z}(G/X,x)\, \bar\varphi_{\Bbb Z}(G|X,y).
\label{Int-Bar-Kappa-Conv-Formula}
\end{equation}
\end{thm}

Equivalent versions of (\ref{Int-Kappa-Decom}),
(\ref{Int-Kappa-Special-Variable}), and
(\ref{Int-Kappa-Conv-Formula}) were stated without proof by Kochol
\cite[p.178]{Kochol3} in a different approach by using chains and in
different notations. The essential difference is that Kochol's
$\kappa_{\Bbb Z}$ is defined formally by
(\ref{Int-Kappa-Conv-Formula}), having no intrinsic combinatorial
meaning as our definition (\ref{kappa-Z-defn}).

\begin{thm}\label{Modular-Thm}
{\rm (a)} The counting function $\kappa(G;p,q)$ {\em
($\bar\kappa(G;p,q)$)} is a polynomial function of positive
(nonnegative) integers $p,q$, having the same degree as the Tutte
polynomial $T_G$, and is independent of the chosen set of distinct
representatives of the cut-Eulerian equivalence classes.

{\rm (b)} {\em Decomposition Formulas:}
\begin{align}
\kappa(G;x,y) &=\sum_{[\rho]\in[{\mathcal O}(G)]}
\kappa_\rho(G;x,y), \label{Mod-Kappa-Decom} \\
\bar\kappa(G;x,y) &=\sum_{[\rho]\in[{\mathcal O}(G)]}
\bar\kappa_\rho(G;x,y). \label{Mod-Bar-Kappa-Decom}
\end{align}

{\rm (c)} {\em Reciprocity Laws:}
\begin{align}
\kappa(G;-x,-y) & = \sum_{[\rho]\in[\mathcal{O}(G)]}
(-1)^{r(G)+|E(C_\rho)|} \bar\kappa_\rho(G;x,y),
\label{Mod-Kappa-Reciprocity}\\
\bar\kappa(G;-x,-y) & = \sum_{[\rho]\in[\mathcal{O}(G)]}
(-1)^{r(G)+|E(C_\rho)|} \kappa_\rho(G;x,y).
\label{Mod-Bar-Kappa-Reciprocity}
\end{align}

{\rm (d)} {\em Specializations:}
\begin{align}
\kappa(G;x,1)&=\tau(G,x),\sp \kappa(G;1,y)=\varphi(G,y);
\label{Mod-Kappa-Special-Variable}\\
\bar\kappa(G;x,-1)&=\bar\tau(G,x),\sp
\bar\kappa(G;-1,y)=\bar\varphi(G,y).
\label{Mod-Bar-Kappa-Special-Variable}
\end{align}

{\rm (e)} {\em Convolution Formulas:}
\begin{align} \kappa(G;x,y) &=\sum_{X\subseteq E}
\tau(G/X,x)\, \varphi(G|X,y), \label{Conv-Formula-Kappa}\\
\bar\kappa(G;x,y) &=\sum_{X\subseteq E} \bar\tau(G/X,x)\,
\bar\varphi(G|X,y). \label{Conv-Formula-Bar-Kappa}
\end{align}
\end{thm}

Equivalent versions of (\ref{Mod-Kappa-Decom}),
(\ref{Mod-Kappa-Special-Variable}), and (\ref{Conv-Formula-Kappa})
were stated without proof by Kochol \cite[p.178]{Kochol3} in a
different approach by using chains in different notations. The
essential difference is that Kochol's $\kappa$ is defined formally
by (\ref{Conv-Formula-Kappa}), having no intrinsic combinatorial
meaning as our definition (\ref{kappa-defn}).

\begin{thm}\label{Rank-Generating-Interpretation}
$\bar\kappa(G;x,y)=R_G(x,y)$.
\end{thm}

The counting definition of $\bar\kappa$ gives rise to an immediate
nontrivial combinatorial-geometric interpretation on the values of
the Tutte polynomial $T_G$ at all positive integers.

\begin{cor}[Combinatorial-Geometric Interpretation]\label{Tutte-Interpretation}
Let $\Rep[\mathcal{O}(G)]$ be any set of distinct representatives of
cut-Eulerian equivalence classes of $\mathcal{O}(G)$. Then
$T_G(p,q)$, where $p,q$ are positive integers, counts the number of
triples
\[
(\rho,f,g),
\]
where $\rho\in\Rep[\mathcal{O}(G)]$, $f$ are nonnegative
integer-valued $p$-tensions of $(G,\rho)$, and $g$ are nonnegative
integer-valued $q$-flows of $(G,\rho)$.
\end{cor}

The interpretation in Corollary~\ref{Tutte-Interpretation} for the
Tutte polynomial $T_G$ is significantly different from
Formula~\eqref{Rank-Generating-Positive}, which can be obtained
directly and easily from definition of Whitney's rank generating
polynomial $R_G$, and which is a reformulation of one main result of
Breuer and Sanyal \cite[p.12]{Breuer-Sanyal}. See
Proposition~\ref{RPQ} and Corollary~\ref{Breuer-Sanyal}.


\section{Preliminaries}

We follow the books \cite{Bollobas1,Bondy-Murty1,Zhang1} for basic
concepts and notations of graphs. Let $G=(V,E)$ be a graph with
possible loops and multiple edges. We call $G$ {\em trivial} if it
has no edges. Let
\[
r(G)=|V|-c(G), \quad n(G)=|E|-r(G),
\]
where $c(G)$ is the number of connected components of $G$. Denote by
$G|X$ the subgraph $(V,X)$ induced by an edge subset $X\subseteq E$.
An {\em orientation} on $G$ is a (multivalued) function
$\varepsilon:V\times E\rightarrow\{-1,0,1\}$ such that (i)
$\varepsilon(v,e)$ has the ordered double-value $\pm1$ or $\mp1$ if
the edge $e$ is a loop at its endvertex $v$ and has a single-value
otherwise, (ii) $\varepsilon(v,e)=0$ if $v$ is not an endvertex of
the edge $e$, and (iii) $\varepsilon(u,e)\varepsilon(v,e)=-1$ if the
edge $e$ has distinct endvertices $u,v$. Pictorially, an orientation
of an edge $e$ can be expressed as an arrow from its one endvertex
$u$ to the other endvertex $v$; such information is encoded by
$\varepsilon(u,e)=1$ if the arrow points away from $u$ and
$\varepsilon(v,e)=-1$ if the arrow points towards $v$. So every edge
has exactly two orientations; each oriented edge contributes exactly
one in-degree and one out-degree. A graph $G$ with an orientation
$\varepsilon$ is referred to a {\em digraph} $(G,\varepsilon)$. We
denote by $\mathcal{O}(G)$ the set of all orientations on $G$.

A {\em cut} of $G$ is a nonempty edge subset $U$ of the form
$[S,S^c]$, where $S$ is a nonempty proper subset of $V$, $S^c:=V-S$,
and $[S,S^c]$ is the set of all edges between vertices of $S$ and
vertices of $S^c$. A {\em bond} is a minimal cut in the sense that
it does not contain any cut properly. Every cut is an edge-disjoint
union of bonds. A {\em directed cut} is a cut $U=[S,S^c]$ together
with an orientation on $U$ such that the arrows of edges are all
from $S$ to $S^c$ or all from $S^c$ to $S$; such an orientation is
called a {\em direction} of $U$. A {\em locally directed cut} is a
cut $U$ together with an orientation $\varepsilon_U$ such that
$(U,\varepsilon_U)$ is an edge-disjoint union of directed bonds;
such an orientation $\varepsilon_U$ is called a {\em local
direction} of cut $U$.

A subgraph $H$ of $G$ is {\em Eulerian} if $H$ has even degree at
every vertex. A {\em circuit} is a minimal, nontrivial (= {\rm
having at least one edge}), Eulerian subgraph in the sense that it
does not contain properly any nontrivial Eulerian subgraph. In fact,
a circuit is just a closed simple path. Every nontrivial Eulerian
subgraph is an edge-disjoint union of circuits. A {\em directed
Eulerian subgraph} is an Eulerian subgraph $H$ together with an
orientation such that at its every vertex the in-degree equals the
out-degree; such an orientation is called a {\em local direction} of
Eulerian graph $H$. The orientation of a directed circuit is called
a {\em direction} of that circuit. Every directed Eulerian subgraph
is an edge-disjoint union of directed circuits.

An orientation $\varepsilon$ on $G$ is said to be {\em acyclic} if
$(G,\varepsilon)$ contains no directed circuit, and is said to be
{\em totally cyclic} if $(G,\varepsilon)$ contains no directed cuts.
We denote by $\mathcal{O}_\AC(G)$ the set of all acyclic
orientations on $G$, and by $\mathcal{O}_\TC(G)$ the set of all
totally cyclic orientations on $G$. Totally cyclic orientations are
also referred to {\em strong orientations}.

Given two orientations $\rho,\sigma\in\mathcal{O}(G)$. We say that
$\rho$ is {\em cut equivalent} to $\sigma$, denoted
$\rho\sim_\CU\sigma$, if the subgraph induced by $E(\rho\neq\sigma)$
is a locally directed cut with orientation either $\rho$ or
$\sigma$. Indeed, $\sim_\CU$ is an equivalence relation on
${\mathcal O}(G)$; see \cite{Chen-I}. If $\rho\in{\mathcal
O}_\AC(G)$ and $\rho\sim_\CU\sigma$, then $\sigma\in{\mathcal
O}_\AC(G)$. Moreover, let $\rho\in{\mathcal O}_\AC(G)$, then
$\rho\sim_\CU\sigma$ if and only if $\rho\sim_\CE\sigma$.
Analogously, $\rho$ is said to be {\em Eulerian equivalent} to
$\sigma$, denoted $\rho\sim_\EU\sigma$, if the subgraph induced by
the edge subset $E(\rho\neq\sigma)$ is a directed Eulerian subgraph
with orientation either $\rho$ or $\sigma$. Indeed, $\sim_\EU$ is an
equivalence relation on ${\mathcal O}(G)$; see \cite{Chen-II}. If
$\rho\in{\mathcal O}_\TC(G)$ and $\rho\sim_\EU\sigma$, then
$\sigma\in{\mathcal O}_\TC(G)$. Moreover, let $\rho\in{\mathcal
O}_\TC(G)$, then $\rho\sim_\EU\sigma$ if and only if
$\rho\sim_\CE\sigma$.

Given two subgraphs $H_i\subseteq G$ with orientations
$\varepsilon_i$, $i=1,2$. The {\em coupling} of $\varepsilon_1$ and
$\varepsilon_2$ is a function $[\varepsilon_1,\varepsilon_2]:
E\rightarrow\{-1,0,1\}$, defined for each edge $e\in E$ (at its one
endvertex $v$) by
\[
[\varepsilon_1,\varepsilon_2](e)= \left\{\begin{array}{rl} 1 &
\mbox{if $e\in E(H_1\cap H_2)$, $\varepsilon_1(v,e)=\varepsilon_2(v,e)$, }\\
-1 & \mbox{if $e\in E(H_1\cap H_2)$, $\varepsilon_1(v,e)\neq\varepsilon_2(v,e)$,}\\
0 & \mbox{otherwise.}
\end{array}\right.
\]
In other words, $[\varepsilon_1,\varepsilon_2](e)=
\varepsilon_1(v,e)\varepsilon_2(v,e)$.

Now let $(G,\varepsilon)$ be a digraph and $A$ an abelian group
throughout the whole paper. There is a {\em boundary operator}
$\partial_\varepsilon: A^E\rightarrow A^V$, defined by
\[
(\partial_\varepsilon f)(v)=\sum_{e\in E}\varepsilon(v,e)f(e),
\]
where $\varepsilon(v,e)$ is counted twice (as $1$ and $-1$) if $e$
is a loop at its unique endvertex $v$. The {\em flow group}
$F(G,\varepsilon;A)$ is $\ker\partial_\varepsilon$, whose elements
are called {\em flows} or $A$-{\em flows} of $(G,\varepsilon)$.
There is a {\em coboundary operator}
$\delta_\varepsilon:A^V\rightarrow A^E$, defined by
\[
(\delta_\varepsilon f)(e)=f(u)-f(v),
\]
where $e$ is an edge whose arrow points from one endvertex $u$ to
the other endvertex $v$. The {\em tension group}
$T(G,\varepsilon;A)$ is $\im\delta_\varepsilon$, whose elements are
called {\em tensions} or $A$-{\em tensions} of $(G,\varepsilon)$.

A function $f\in A^E$ is {\em nowhere-zero} if $f(e)\neq 0$ for all
$e\in E$, and is a {\em $q$-function} if $A=\Bbb R$ and $|f(e)|<q$
for all $e\in E$. Let $\tau(G,q)$ ($\varphi(G,q)$) denote the number
of nowhere-zero $A$-tensions ($A$-flows) of $(G,\varepsilon)$ with
$|A|=q$. Let $\tau_{\Bbb Z}(G,q)$ ($\varphi_{\Bbb Z}(G,q)$) denote
the number of integer-valued nowhere-zero $q$-tensions ($q$-flows)
of $(G,\varepsilon)$. Let $\tau_\varepsilon(G,q)$
($\varphi_\varepsilon(G,q)$) denote the number of integer-valued
tensions (flows) $f$ of $(G,\varepsilon)$ such that $0<f(e)<q$ for
all $e\in E$. It is well known that $\tau$, $\varphi$, $\tau_{\Bbb
Z}$, $\varphi_{\Bbb Z}$, $\tau_\varepsilon$, $\varphi_\varepsilon$
are polynomial functions of positive integers $q$, and that $\tau$,
$\varphi$, $\tau_{\Bbb Z}$, $\varphi_{\Bbb Z}$ are independent of
the chosen orientation $\varepsilon$, and that $\tau$, $\varphi$ are
further independent of the group structure of $A$. The polynomial
$\tau$ ($\tau_{\Bbb Z}$, $\tau_\varepsilon$) is referred to the {\em
modular (integral, local) tension polynomial} of $G$, and $\varphi$
($\varphi_{\Bbb Z}$, $\varphi_\varepsilon$) to the {\em modular
(integral, local) flow polynomial}; see \cite{Chen-I,Chen-II} in
details.

There are polynomials {\em dual} to $\tau$, $\varphi$, $\tau_{\Bbb
Z}$, $\varphi_{\Bbb Z}$, $\tau_\varepsilon$, $\varphi_\varepsilon$
respectively. Let $\bar\tau_\varepsilon(G,q)$
($\bar\varphi_\varepsilon(G,q)$) denote the number of integer-valued
tensions (flows) $f$ of $(G,\varepsilon)$ such that $0\leq f(e)\leq
q$ for all $e\in E$. Let $\bar\tau_{\Bbb Z}(G,q)$
($\bar\varphi_{\Bbb Z}(G,q)$) denote the number of ordered pairs
$(\rho, f)$, where $\rho\in\mathcal{O}_\AC(G)$
($\rho\in\mathcal{O}_\TC(G)$) and $f$ is an integer-valued tension
(flow) of $(G,\rho)$ such that $0\leq f(e)\leq q$ for all $e\in E$.
Let $\Rep[\mathcal{O}_\AC(G)]$ ($\Rep[\mathcal{O}_\TC(G)]$) be any
set of distinct representatives of cut (Eulerian) equivalence
classes of $\mathcal{O}_\AC(G)$ ($\mathcal{O}_\TC(G)$). Let
$\bar\tau(G,q)$ ($\bar\varphi(G,q)$) denote the number of ordered
pairs $(\rho, f)$, where $\rho\in\Rep[\mathcal{O}_\AC(G)]$
($\Rep[\mathcal{O}_\TC(G)]$) and $f$ is an integer-valued tension
(flow) of $(G,\rho)$ such that $0\leq f(e)\leq q$ for all $e\in E$.
It turns out that $\bar\tau_\varepsilon$, $\bar\varphi_\varepsilon$,
$\bar\tau_{\Bbb Z}$, $\bar\varphi_{\Bbb Z}$, $\bar\tau$,
$\bar\varphi$ are polynomial functions of nonnegative integers $q$,
and are independent of the chosen set of distinct representatives.
Moreover, these polynomials represent the same polynomials
$\tau_\varepsilon$, $\varphi_\varepsilon$, $\tau_{\Bbb Z}$,
$\varphi_{\Bbb Z}$, $\tau$, $\varphi$ respectively, up to sign and
change of the variable; see \cite{Chen-I,Chen-II} in details.

Let $A,B$ be abelian groups. The {\em tension-flow group} of digraph
$(G,\varepsilon)$ is the abelian group
\[
\Omega(G,\varepsilon;A,B):=T(G,\varepsilon;A)\times
F(G,\varepsilon;B),
\]
whose elements are called {\em tension-flows}. A tension-flow
$(f,g)$ is said to be {\em nowhere-zero} if
\[
(f(e),g(e))\neq(0,0) \quad \mbox{for all $e\in E$,}
\]
and to be {\em complementary} if $\ker f=\supp g$, where
\[
\ker f=\{e\in E:f(e)=0\}, \quad \supp g=\{e\in E:g(e)\neq 0\}.
\]
We denote by $K(G,\varepsilon;A,B)$ the set of all complementary
tension-flows of $(G,\varepsilon)$. For simplicity, we write
$F(G,\varepsilon)$ for $F(G,\varepsilon;{\Bbb R})$,
$T(G,\varepsilon)$ for $T(G,\varepsilon;{\Bbb R})$, and
$\Omega(G,\varepsilon)$ for $\Omega(G,\varepsilon;{\Bbb R})$. It is
well known that $T(G,\varepsilon)$ and $F(G,\varepsilon)$ are
orthogonal complements in the Euclidean space ${\Bbb R}^E$. For
positive integer $p,q$, a real-valued tension-flow $(f,g)$ is called
a {\em $(p,q)$-tension-flow} if $|f(e)|<p$ and $|g(e)|<q$ for all
$e\in E$.

The present paper is conceptually rely on the Ehrhart theory of
lattice polytopes and polyhedra for which we refer to
\cite{Beck-Robins,Lattice-points,Ehrhart-poly,Stanley2}. Let $P$ be
a bounded lattice polyhedron (finite disjoint union of relatively
open convex lattice polytopes) in the Euclidean $d$-space ${\Bbb
R}^d$. Let $qP=\{qx:x\in P\}$ for positive integers $q$. The
counting function
\begin{equation}\label{Ehrhart-Poly-Defn}
L(P,q):=\#({\Bbb Z}^d\cap qP)
\end{equation}
is a polynomial function of positive integers $q$ of degree $\dim
P$, called the {\em Ehrhart polynomial} of $P$. If $P$ is a
relatively open lattice polytope and $\bar P$ its closure, then
$L(P,t)$ and $L(\bar P,t)$ satisfy the {\em Reciprocity Law:}
\begin{equation}\label{Ehrhart-Reciprocity-Law}
L(P,-t)=(-1)^{\dim P}L(\bar P,t).
\end{equation}
Moreover, $L(P,0)=(-1)^{\dim P}$, $L(\bar P,0)=1$.

\section{Integral complementary polynomials}

Recall that a real-valued tension-flow
$(f,g)\in\Omega(G,\varepsilon)$ is {\em complementary} if and only
if $f(e)g(e)=0$ and $f(e)+g(e)\neq 0$ for all $e\in E$. We denote by
$K(G,\varepsilon)$ the set of all real-valued complementary
tension-flows of $(G,\varepsilon)$. We introduce the {\em
complementary polyhedron}
\begin{equation}\label{Delta-CTF-Defn}
\Delta_\CTF(G,\varepsilon)=\{(f,g)\in K(G,\varepsilon):
0<|f(e)+g(e)|<1, e\in E\}
\end{equation}
which is a bounded relatively open non-convex polyhedron, the {\em
complementary polytope}
\begin{equation}\label{Positive-Delta-CTF-Defn}
\Delta^+_\CTF(G,\varepsilon)=\{(f,g)\in K(G,\varepsilon):
0<f(e)+g(e)<1, e\in E\}
\end{equation}
which is a relatively open convex $0$-$1$ polytope, and the
relatively open convex polytope (with respect to an orientation
$\rho$)
\begin{equation}
\Delta^\rho_\CTF(G,\varepsilon)
=\{(f,g)\in\Delta_\CTF(G,\varepsilon):[\rho,\varepsilon](e)(f+g)(e)>0,e\in
E\}.
\end{equation}
For positive integers $p,q$, recall the counting function
\begin{align}
\kappa_{\Bbb Z}(G;p,q) &=\#\{\mbox{${\Bbb Z}$-valued complementary
$(p,q)$-tension-flows of
$(G,\varepsilon)$}\}.\label{Int-Bar-Kappa-Defn*}
\end{align}
For nonnegative integers $p,q$, recall the counting function
\begin{align}
\bar\kappa_{\Bbb Z}(G;p,q) &=\#\{(\rho,f,g):
\mbox{$\rho\in\mathcal{O}(G)$, $(f,g) \in
\Omega(G,\varepsilon)\cap({\Bbb Z}\times{\Bbb Z})^E$}\nonumber\\
&\hspace{8mm} \mbox{such that $0\leq f(e)\leq p$, $0\leq g(e)\leq q$
for $e\in E$}\}. \label{Int-Bar-Kappa-Defn*}
\end{align}
Let
$\widetilde{\Delta}_\CTF(G,\varepsilon)=\sum_{\rho\in\mathcal{O}(G)}
\bar\Delta^\rho_\CTF(G,\varepsilon)$ be a topological sum defined as
a disjoint union of copies of closures
$\bar\Delta^\rho_\CTF(G,\varepsilon)$, one copy for each
$\rho\in{\mathcal O}(G)$. Then $\bar\kappa_{\Bbb Z}(G;p,q)$ counts
the number of lattice points of
$(p,q)\widetilde{\Delta}_\CTF(G,\varepsilon)$.

We introduce the following two special directed subgraphs:
\begin{align*}
B_\varepsilon &=\mbox{\rm union of directed bonds of
$(G,\varepsilon)$},\\
C_\varepsilon &=\mbox{\rm union of directed circuits of
$(G,\varepsilon)$}.
\end{align*}
It is clear that $B_\varepsilon$ is acyclic, $C_\varepsilon$ is
totally cyclic, and their edge sets are disjoint. The following
lemma is a special case of Minty's Colored Arc Lemma \cite{Minty1}.

\begin{lemma}\label{EBC}
$E=E(B_\varepsilon)\sqcup E(C_\varepsilon)$ {\rm (disjoint union)}.
\end{lemma}
\begin{proof}
Since each directed circuit of $(G,\varepsilon)$ is edge-disjoint
from any directed bond of $(G,\varepsilon)$, it is clear that the
edge sets of $B_\varepsilon$ and $C_\varepsilon$ are disjoint. To
see that $E(B_\varepsilon)=E-E(C_\varepsilon)$, consider the
quotient digraph $G/{C_\varepsilon}$ obtained from $(G,\varepsilon)$
by contracting the edges of $C_\varepsilon$. Clearly,
$G/C_\varepsilon$ is acyclic and the edge set of $G/{C_\varepsilon}$
can be identified as $E-E(C_\varepsilon)$. It is clear that
$E(G/{C_\varepsilon})$ can be written as a union of directed bonds
(not necessarily edge-disjoint). For ease of discussion, we call the
inverse operation of contracting an edge as a {\em blow-up} at a
vertex. It is easy to see that blow-up does not change directed
bonds. So every directed bond of $G/C_\varepsilon$ is preserved into
a directed bond in $(G,\varepsilon)$ when the edges of
$C_\varepsilon$ are blew up from $G/C_\varepsilon$. So
$E-E(C_\varepsilon)$ is a union of directed bonds. Hence
$E(B_\varepsilon)=E-E(C_\varepsilon)$.
\end{proof}

Recall that the tension polytope and the flow polytope of digraph
$(G,\varepsilon)$, which are relatively open 0-1 polytopes (see
\cite{Chen-I,Chen-II}), are defined respectively as
\[
\Delta^+_\TN(G,\varepsilon)=\{f\in T(G,\varepsilon): 0<f(e)<1, e\in
E\},
\]
\[
\Delta^+_\FL(G,\varepsilon)=\{f\in F(G,\varepsilon): 0<f(e)<1, e\in
E\}.
\]
The complementary polytope $\Delta^+_\CTF(G,\varepsilon)$ can be
decomposed into a product of a face of $\Delta^+_\TN(G,\varepsilon)$
and a face of $\Delta^+_\FL(G,\varepsilon)$. In fact,
\begin{equation}\label{Positive-Delta-Product-Decom}
\Delta^{+}_{\CTF}(G,\varepsilon) =\Delta^+_{\TN}(G,B_\varepsilon)
\times \Delta^+_{\FL}(G,C_\varepsilon),
\end{equation}
where
\[
\Delta^+_{\TN}(G,B_\varepsilon) =\{f\in T(G,\varepsilon):
0<f(e)<1,e\in B_\varepsilon, f|_{C_\varepsilon}=0\},
\]
\[
\Delta^+_{\FL}(G,C_\varepsilon) =\{g\in F(G,\varepsilon):
0<g(e)<1,e\in C_\varepsilon, g|_{B_\varepsilon}=0\}.
\]
To find the relationship holding among the associated polyhedra and
polytopes, we need the involution
\[
P_{\rho,\sigma}:{\Bbb R}^E\rightarrow {\Bbb R}^E, \quad f\mapsto
[\rho,\sigma]\,f
\]
associated with two orientations $\rho,\sigma\in{\mathcal O}(G)$. It
induces an involution
\[
P_{\rho,\sigma}:({\Bbb R}\times{\Bbb R})^E\rightarrow ({\Bbb
R}\times{\Bbb R})^E, \quad (f,g)\mapsto
([\rho,\sigma]\,f,[\rho,\sigma]\,g).
\]

\begin{lemma}\label{Delta-CTF-Decomposition}
{\rm (a)} $\Delta_{\CTF}(G,\varepsilon)
 =\bigsqcup_{\rho\in{\mathcal O}(G)}
\Delta^\rho_\CTF(G,\varepsilon)$.

{\rm (b)} $\Delta^\rho_\CTF(G,\varepsilon) =
P_{\rho,\varepsilon}\Delta^+_\CTF(G,\rho)$.

{\rm (c)} $(p,q)\Delta^+_\CTF(G,\varepsilon)
=p\Delta^+_\TN(G,B_\varepsilon) \times
q\Delta^+_\FL(G,C_\varepsilon)$.

{\rm (d)} $p\Delta^+_\TN(G,B_\varepsilon) \simeq
p\Delta^+_\TN(G/C_\varepsilon,\varepsilon)$,
$q\Delta^+_\FL(G,C_\varepsilon) \simeq
q\Delta^+_\FL(G|C_\varepsilon,\varepsilon)$; the isomorphisms send
lattice points to lattice points.
\end{lemma}
\begin{proof}
(a) The right-hand side is clearly contained in the left-hand side,
since each element in the right-hand side is a real-valued
complementary $(1,1)$-tension-flow. Conversely, for each real-valued
complementary $(1,1)$-tension-flow $(f,g)$ of $(G,\varepsilon)$, let
$\rho$ be the orientation given by
\[
\rho(v,e)=\left\{\begin{array}{rl} \varepsilon(v,e) & \mbox{if
$f(e)+g(e)>0$,}\\
-\varepsilon(v,e) & \mbox{if $f(e)+g(e)<0$,}
\end{array}\right.
\]
where $v$ is an endvertex of the edge $e\in E$. It is clear that
$[\rho,\varepsilon](f+g)(e)>0$ for all $e\in E$. Hence $(f,g)\in
\Delta^\rho_{\CTF}(G,\varepsilon)$.

(b) Let $(f,g)$ be a complementary $(1,1)$-tension-flow of
$(G,\varepsilon)$. Then $(f,g)\in\Delta^\rho_\CTF(G,\varepsilon)$ if
and only if $[\rho,\varepsilon](f+g)(e)>0$ for $e\in E$; i.e., if
and only if $[\rho,\varepsilon](f,g)\in\Delta^+_\CTF(G,\rho)$, since
$[\rho,\varepsilon](f,g):=([\rho,\varepsilon]f,[\rho,\varepsilon]g)$
is a tension-flow of $(G,\rho)$; and equivalently, $(f,g)\in
P_{\rho,\varepsilon}\Delta^+_\CTF(G,\rho)$, since
$P_{\rho,\varepsilon}(f,g):=[\rho,\varepsilon](f,g)$ and
$P_{\rho,\varepsilon}$ is an involution.

(c) Trivial.

(d) The bijection between $\Delta^+_\TN(G,B_\varepsilon)$ and
$\Delta^+_\TN(G/C_\varepsilon,\varepsilon)$ is given by $f\mapsto
f|_{B_\varepsilon}$; and the bijection between
$\Delta^+_\FL(G,C_\varepsilon)$ and
$\Delta^+_\FL(G|C_\varepsilon,\varepsilon)$ is given by $f\mapsto
f|_{C_\varepsilon}$.
\end{proof}

\textsc{Remark.} The polytope $\Delta^+_\TN(G,B_\varepsilon)$ cannot
be identified to the polytope
$\Delta^+_\TN(G|B_\varepsilon,\varepsilon)$, since a tension of
digraph $(G|B_\varepsilon,\varepsilon)$ can not be viewed as a
tension of $(G,\varepsilon)$ that vanishes on the edge subset of
$C_\varepsilon$.

\begin{prop}\label{Local-Thm}
{\rm (a)} The counting function $\kappa_\varepsilon(G;p,q)$ {\em
($\bar\kappa_\varepsilon(G;p,q)$)} is a polynomial function of
positive (nonnegative) integers $p,q$ of degree
$r(G/C_\varepsilon)+n(G|C_\varepsilon)$.

{\rm (b)} {\em Product Decomposition:}
\begin{equation}\label{Local-Kappa-Product}
\kappa_\varepsilon(G;x,y) =\tau_\varepsilon(G/C_\varepsilon,x)\,
\varphi_\varepsilon(G|C_\varepsilon,y),
\end{equation}
\begin{equation}\label{Local-Bar-Kappa-Product}
\bar\kappa_\varepsilon(G;x,y)
=\bar\tau_\varepsilon(G/C_\varepsilon,x)\,
\bar\varphi_\varepsilon(G|C_\varepsilon,y).
\end{equation}
Moreover, $\tau_\varepsilon(G/C_\varepsilon,x)$,
$\varphi_\varepsilon(G|C_\varepsilon,y)$,
$\bar\tau_\varepsilon(G/C_\varepsilon,x)$, and
$\bar\varphi_\varepsilon(G|C_\varepsilon,y)$ are the Ehrhart
polynomials of lattice polytopes
$\Delta^+_\TN(G/C_\varepsilon,\varepsilon)$,
$\Delta^+_\FL(G|C_\varepsilon,\varepsilon)$,
$\bar\Delta^+_\TN(G/C_\varepsilon,\varepsilon)$, and
$\bar\Delta^+_\FL(G|C_\varepsilon, \varepsilon)$ respectively.

{\rm (c)} {\em Reciprocity Law:}
\begin{equation}\label{Local-Reciprocity-Law}
\kappa_\varepsilon(G;-x,-y)
=(-1)^{r(G)+|E(C_\varepsilon)|}\bar\kappa_\varepsilon(G;x,y).
\end{equation}

{\rm (d)} {\em Specializations:}
\begin{equation}\label{Local-Kappa-Special-Variable-x}
\kappa_\varepsilon(G;x,1) =\tau_\varepsilon(G,x),
\end{equation}
\begin{equation}\label{Local-Kappa-Special-Variable-y}
\kappa_\varepsilon(G;1,y) =\varphi_\varepsilon(G,y);
\end{equation}
\begin{equation}\label{Local-Kappa-Special-Variable-Tension}
\bar\kappa_\varepsilon(G;x,-1)
=(-1)^{|E(C_\varepsilon)|}\bar\tau_\varepsilon(G,x),
\end{equation}
\begin{equation}\label{Local-Kappa-Special-Variable-Flow}
\bar\kappa_\varepsilon(G;-1,y)
=(-1)^{|E(B_\varepsilon)|}\bar\varphi_\varepsilon(G,y).
\end{equation}
\end{prop}
\begin{proof}
(a) It is an immediate consequence of (b).

(b) It follows from the product decomposition
(\ref{Positive-Delta-Product-Decom}) and the identifications of the
lattice polytopes in Lemma~\ref{Delta-CTF-Decomposition}(d).

(c) It follows from the Reciprocity Law
(\ref{Ehrhart-Reciprocity-Law}) of Ehrhart polynomials and the
relation
\[
r(G/C_\varepsilon)+n(G|C_\varepsilon) =|E(C_\varepsilon)|
+2\,c(G|C_\varepsilon) -r(G) -2\,c(G).
\]

(d) Let $p,q$ be positive integers. Let $(f,g)$ be a positive
integer-valued $(p,1)$-tension-flow of $(G,\varepsilon)$. Since
$0<g(e)<1$ is impossible for any $e\in E$, we must have
$E(C_\varepsilon)=\emptyset$ and $E(B_\varepsilon)=E$; i.e., the
orientation $\varepsilon$ must be acyclic. So $(f,g)$ is reduced to
a positive integer-valued $p$-tension $f$ of $(G,\varepsilon)$. We
thus have $\kappa_\varepsilon(G;p,1) =\tau_\varepsilon(G,p)$. Note
that $\tau_\varepsilon$ is not the zero polynomial if and only if
$G$ is loopless and contains some edges.

Analogously, let $(f,g)$ be a positive integer-valued
$(1,q)$-tension-flow of $(G,\varepsilon)$. Since $0<f(e)<1$ is
impossible for any $e\in E$, we must have
$E(B_\varepsilon)=\emptyset$ and $E(C_\varepsilon)=E$; i.e., the
orientation $\varepsilon$ must be totally cyclic. So $(f,g)$ is
reduced to a positive integer-valued $q$-flow $g$ of
$(G,\varepsilon)$. We thus have $\kappa_\varepsilon(G;1,q)
=\varphi_\varepsilon(G,q)$. Note that $\varphi_\varepsilon$ is not
the zero polynomial if and only if $G$ is bridgeless and contains
some edges. We finish the proof of
(\ref{Local-Kappa-Special-Variable-x}) and
(\ref{Local-Kappa-Special-Variable-y}).

Now applying
(\ref{Local-Reciprocity-Law})--(\ref{Local-Kappa-Special-Variable-y}),
we have
\begin{align*}
\bar\kappa_\varepsilon(G;x,-1) &= (-1)^{r(G)+|E(C_\varepsilon)|}
\kappa_\varepsilon(G;-x,1) \\
&= (-1)^{r(G)+|E(C_\varepsilon)|}
\tau_\varepsilon(G,-x)\\
&=(-1)^{|E(C_\varepsilon)|} \bar\tau_\varepsilon(G,x);
\end{align*}
\begin{align*}
\bar\kappa_\varepsilon(G;-1,y) &= (-1)^{r(G)+|E(C_\varepsilon)|}
\kappa_\varepsilon(G;1,-y) \\
&= (-1)^{r(G)+|E(C_\varepsilon)|}
\varphi_\varepsilon(G,-y)\\
&=(-1)^{|E(B_\varepsilon)|} \bar\varphi_\varepsilon(G,y).
\end{align*}
The last two equality in both follow from the Reciprocity Laws
\begin{equation}
\tau_\varepsilon(G,-x) =(-1)^{r(G)}\bar\tau_\varepsilon(G,x), \quad
\varphi_\varepsilon(G,-y) =(-1)^{n(G)}\bar\varphi_\varepsilon(G,y)
\end{equation}
respectively; see (1.9) of Theorem~1.2 in \cite{Chen-I} and (1.7) of
Theorem~1.1 in \cite{Chen-II}. We finish the proof of
(\ref{Local-Kappa-Special-Variable-Tension}) and
(\ref{Local-Kappa-Special-Variable-Flow}).
\end{proof}

\textsc{Proof of Theorem~\ref{Integral-Thm}} \vspace{1ex}

(a) The polynomiality follows from (b) and (c) of the same theorem,
and from (a) and (b) of Proposition~\ref{Local-Thm}. The
independence of the chosen orientation $\varepsilon$ follows from
the bijection $P_{\rho,\varepsilon}:
(p,q)\Delta_\CTF(G,\varepsilon)\rightarrow
(p,q)\Delta_\CTF(G,\rho)$, which sends lattice points to lattice
points.

(b) The decomposition (\ref{Int-Kappa-Decom}) follows from the
disjoint composition of the complementary polyhedron
$\Delta_\CTF(G,\varepsilon)$ (see
Lemma~\ref{Delta-CTF-Decomposition}(a)) and the property that the
involution $P_{\rho,\varepsilon}$ sends lattice points to lattice
points bijectively (see Lemma~\ref{Delta-CTF-Decomposition}(b)). The
decomposition (\ref{Int-Bar-Kappa-Decom}) is trivial by definition
of $\bar\kappa_{\Bbb Z}$.

(c) It follows from the Reciprocity Law
(\ref{Local-Reciprocity-Law}).

(d) Let $p,q$ be integers larger than or equal to $2$. Consider an
integer-valued complementary $(p,1)$-tension-flow $(f,g)$ of
$(G,\varepsilon)$. Since $|g|<1$ implies $g=0$, then $f$ is a
nowhere-zero $p$-tension of $(G,\varepsilon)$, and $(f,0)$ is
identified to the nowhere-zero $p$-tension $f$ of $(G,\varepsilon)$.
Thus $\kappa_{\Bbb Z}(G;p,1)=\tau_{\Bbb Z}(G,p)$. Likewise, an
integer-valued complementary $(1,q)$-tension-flow $(f,g)$ of
$(G,\varepsilon)$ implies that $f=0$ and $g$ is a nowhere-zero
$q$-flow of $(G,\varepsilon)$. Hence $\kappa_{\Bbb
Z}(G;1,q)=\varphi_{\Bbb Z}(G,q)$. We finish the proof of
(\ref{Int-Kappa-Special-Variable}).

Now consider the case of $y=1$ and the case of $x=1$ in
(\ref{Int-Kappa-Decom}). Applying Proposition~\ref{Local-Thm}(d),
(\ref{Int-Kappa-Decom}) becomes
\begin{align}
\tau_{\Bbb Z}(G,x) &= \sum_{\rho\in\mathcal{O}_\AC(G)}
\tau_\rho(G,x),\label{Int-Tau-Decom}\\
\varphi_{\Bbb Z}(G,y) &= \sum_{\rho\in\mathcal{O}_\TC(G)}
\varphi_\rho(G,y). \label{Int-Phi-Decom}
\end{align}
We have recovered (1.7) of Theorem~1.2 in \cite[p.428]{Chen-I} and
(1.8) of Theorem~1.1(b) in \cite[p.754]{Chen-II}; see also
\cite{Kochol1,Kochol2} for equivalent versions. The dual polynomials
$\bar\tau_{\Bbb Z}$ and $\bar\varphi_{\Bbb Z}$ have similar
decompositions by their definitions:
\begin{align}
\bar\tau_{\Bbb Z}(G,x) &= \sum_{\rho\in\mathcal{O}_\AC(G)}
\bar\tau_\rho(G,x),\label{Int-Bar-Tau-Decom}\\
\bar\varphi_{\Bbb Z}(G,y) &= \sum_{\rho\in\mathcal{O}_\TC(G)}
\bar\varphi_\rho(G,y). \label{Int-Bar-Phi-Decom}
\end{align}
Analogously, consider the case of $y=-1$ and the case of $x=-1$ in
(\ref{Int-Bar-Kappa-Decom}). We have
\begin{align*}
\bar\kappa_{\Bbb Z}(G;x,-1) &= \sum_{\rho\in\mathcal{O}(G)}
(-1)^{r(G)+|E(C_\rho)|} \kappa_\rho(G;-x,1) \\
&= \sum_{\rho\in\mathcal{O}_\AC(G)}
(-1)^{r(G)} \tau_\rho(G,-x) \\
&= \sum_{\rho\in\mathcal{O}_\AC(G)}
\bar\tau_\rho(G,x)=\bar\tau_{\Bbb Z}(G,x),
\end{align*}
where the first equality follows from
(\ref{Int-Bar-Kappa-Reciprocity}), the second one follows from
(\ref{Local-Kappa-Special-Variable-x}) and the equivalence of
$\tau_\rho\neq 0$ and $\rho\in{\mathcal O}_\AC(G)$, and the third
one follows from (\ref{Int-Bar-Tau-Decom}). Using $r(G)+n(G)=|E|$
and $|E|=|E(B_\rho)|+|E(C_\rho)|$, a similar argument implies
\begin{align*}
\bar\kappa_{\Bbb Z}(G;-1,y) &= \sum_{\rho\in\mathcal{O}(G)}
(-1)^{n(G)+|E(B_\rho)|} \kappa_\rho(G;1,-y) \\
&= \sum_{\rho\in\mathcal{O}_\TC(G)}
(-1)^{n(G)} \varphi_\rho(G,-y) \\
&= \sum_{\rho\in\mathcal{O}_\TC(G)}
\bar\varphi_\rho(G,y)=\bar\varphi_{\Bbb Z}(G,y),
\end{align*}
where the first equality follows from
(\ref{Int-Bar-Kappa-Reciprocity}), the second from
(\ref{Local-Kappa-Special-Variable-y}) and the equivalence of
$\varphi_\rho\neq 0$ and $\rho\in{\mathcal O}_\TC(G)$, and the last
one from (\ref{Int-Bar-Phi-Decom}). We finish the proof of
(\ref{Int-Bar-Kappa-Special-Variable}).

(e) Finally, let $X\subseteq E$ be an edge subset. For each
orientation $\rho\in\mathcal{O}(G)$ such that $E(C_\rho)=X$ (such an
orientation my not exist), let $\rho_{|X}$ denote the restriction of
$\rho$ on $G|X$, and let $\rho_{/X}$ denote the induced orientation
on $G/X$. Note that $\rho_{|X}$ is totally cyclic on $G|X$ and
$\rho_{/X}$ is acyclic on $G/X$. Then
\begin{align*}
\kappa_{\Bbb Z}(G;x,y) &= \sum_{X\subseteq E}
\sum_{\rho\in\mathcal{O}(G)\atop C_\rho=X}
\tau_\rho(G/C_\rho,x)\, \varphi_\rho(G|C_\rho,y)\\
&= \sum_{X\subseteq E} \sum_{\rho\in\mathcal{O}(G)\atop C_\rho=X}
\tau_{\rho_{|X}}(G/X,x)\,
\varphi_{\rho_{/X}}(G|X,y)\\
&= \sum_{X\subseteq E} \sum_{\rho\in\mathcal{O}_\AC(G/X)\atop
\sigma\in\mathcal{O}_\TC(G|X)} \tau_\rho(G/X,x)\,
\varphi_\sigma(G|X,y),
\end{align*}
where the first equality follows from the equations
(\ref{Int-Kappa-Decom}) and (\ref{Local-Kappa-Product}). Apply
(\ref{Int-Tau-Decom}) to graph $G/X$ and (\ref{Int-Phi-Decom}) to
graph $G|X$ in the above last equality; we obtain
(\ref{Int-Kappa-Conv-Formula}) immediately. Analogously, the
equations (\ref{Int-Bar-Kappa-Decom}) and
(\ref{Local-Bar-Kappa-Product}) imply that
\begin{align*}
\bar\kappa_{\Bbb Z}(G;x,y) &= \sum_{X\subseteq E}
\sum_{\rho\in\mathcal{O}(G)\atop C_\rho=X}
\bar\tau_\rho(G/C_\rho,x)\,
\bar\varphi_\rho(G|C_\rho,y)\\
&= \sum_{X\subseteq E} \sum_{\rho\in\mathcal{O}(G)\atop C_\rho=X}
\bar\tau_{\rho_{|X}}(G/X,x)\,
\bar\varphi_{\rho_{/X}}(G|X,y)\\
&= \sum_{X\subseteq E} \sum_{\rho\in\mathcal{O}_\AC(G/X)\atop
\sigma\in\mathcal{O}_\TC(G|X)} \bar\tau_\rho(G/X,x)\,
\bar\varphi_\sigma(G|X,y),
\end{align*}
Applying (\ref{Int-Bar-Tau-Decom}) to $G/X$ and
(\ref{Int-Bar-Phi-Decom}) to $G|X$, we obtain
(\ref{Int-Bar-Kappa-Conv-Formula}). \hfill{$\Box$} \vspace{2ex}

The polynomials $\kappa_{\Bbb Z}$ and $\bar\kappa_{\Bbb Z}$ have the
following particular combinatorial interpretations at some special
integers.

\begin{cor}\label{Integral-Kappa-Special}
{\rm (a)} $\bar\kappa_{\Bbb Z}(G;0,0) =|\mathcal{O}(G)|$,
\begin{align*}
|\kappa_{\Bbb Z}(G;1,0)| &=\bar\kappa_{\Bbb Z}(G;-1,0)
= |\mathcal{O}_\TC(G)|,\\
|\kappa_{\Bbb Z}(G;0,1)| &=\bar\kappa_{\Bbb Z}(G;0,-1) =
|\mathcal{O}_\AC(G)|, \\
\kappa_{\Bbb Z}(G;1,1) &=\bar\kappa_{\Bbb Z}(G;-1,-1)=0.
\end{align*}

{\rm (b)} Let $\mathcal{O}_\CU(G)$, $\mathcal{O}_\EU(G)$, and
$\mathcal{O}_\CE(G)$ be the sets of orientations $\rho$ such that
$(G,\rho)$ is a locally directed cut, a directed Eulerian graph, and
an edge-disjoint union of a locally directed cut and a directed
Eulerian subgraph, respectively. Then
\begin{align*}
\kappa_{\Bbb Z}(G;2,1) &= |\bar\kappa_{\Bbb Z}(G;-2,-1)| =
|\mathcal{O}_\CU(G)|,\\
\kappa_{\Bbb Z}(G;1,2) &= |\bar\kappa_{\Bbb Z}(G;-1,-2)| =
|\mathcal{O}_\EU(G)|,\\
\kappa_{\Bbb Z}(G;2,2) &= |\mathcal{O}_\CE(G)|.
\end{align*}

{\rm (c)} For each orientation $\rho$, let $[\rho]_\CU$,
$[\rho]_\EU$, and $[\rho]_\CE$ denote the equivalence classes of
${\mathcal O}(G)$ under the cut, Eulerian, and cut-Eulerian
equivalence relations respectively. Then
\begin{align*}
\bar\kappa_{\Bbb Z}(G;1,0) & = \sum_{\rho\in{\mathcal O}(G)} \#[\rho]_\CU,\\
\bar\kappa_{\Bbb Z}(G;0,1) & = \sum_{\rho\in{\mathcal O}(G)} \#[\rho]_\EU,\\
\bar\kappa_{\Bbb Z}(G;1,1) & = \sum_{\rho\in{\mathcal O}(G)}
\#[\rho]_\CE.
\end{align*}
\end{cor}
\begin{proof}
(a) Since $\bar\kappa_\rho(G;0,0)=1$ for $\rho\in{\mathcal O}(G)$,
then by (\ref{Int-Bar-Kappa-Decom}) we have $\bar\kappa_{\Bbb
Z}(G;0,0)=|\mathcal{O}(G)|$.

Since $\tau_\rho(G,0)=(-1)^{r(G)}$ for $\rho\in\mathcal{O}_\AC(G)$,
then by (\ref{Int-Tau-Decom}) we have $\tau_{\Bbb Z}(G,0)
=(-1)^{r(G)}|\mathcal{O}_\AC(G)|$. Since
$\varphi_\rho(G,0)=(-1)^{n(G)}$ for $\rho\in\mathcal{O}_\TC(G)$,
then by (\ref{Int-Phi-Decom}) we have $\varphi_{\Bbb Z}(G,0)
=(-1)^{n(G)}|\mathcal{O}_\TC(G)|$. According to
Theorem~\ref{Integral-Thm}(d), we see that $\kappa_{\Bbb Z}(G;1,0)
=(-1)^{n(G)}|\mathcal{O}_\TC(G)|$ and $\kappa_{\Bbb Z}(G;0,1)
=(-1)^{r(G)}|\mathcal{O}_\AC(G)|$. Notice the Reciprocity Laws on
$\tau_{\Bbb Z},\bar\tau_{\Bbb Z}$ and on $\varphi_{\Bbb
Z},\bar\varphi_{\Bbb Z}$. Again according to
Theorem~\ref{Integral-Thm}(d), we see that $\bar\kappa_{\Bbb
Z}(G;-1,0) =|\mathcal{O}_\TC(G)|$ and $\bar\kappa_{\Bbb Z}(G;0,-1)
=|\mathcal{O}_\AC(G)|$.

Since $\kappa_\rho(G;1,1)=0$ for $\rho\in{\mathcal O}(G)$, then
$\bar\kappa_\rho(G;-1,-1)=0$ for $\rho\in{\mathcal O}(G)$ by
(\ref{Local-Reciprocity-Law}). Hence by (\ref{Int-Kappa-Decom}) and
(\ref{Int-Bar-Kappa-Decom}), we have $\kappa_{\Bbb
Z}(G;1,1)=\bar\kappa_{\Bbb Z}(G;-1,-1)=0$.

(b) If $f$ is an integer-valued tension of a digraph $(G,\rho)$ such
that $0<f(e)<2$ for all $e\in E$, then $f\equiv 1$. This means that
$(G,\rho)$ is an edge-disjoint union of directed bonds, i.e., a
locally directed cut. Thus $\tau_{\Bbb Z}(G,2)=|{\mathcal
O}_\CU(G)|$ by (\ref{Int-Kappa-Decom}), and subsequently,
$\bar\tau_{\Bbb Z}(G,-2)=(-1)^{r(G)}|{\mathcal O}_\CU(G)|$.
According to Theorem~\ref{Integral-Thm}(d), we have $\kappa_{\Bbb
Z}(G;2,1)=|\mathcal{O}_\CU(G)|$ and $\bar\kappa_{\Bbb
Z}(G;-2,-1)=(-1)^{r(G)}|\mathcal{O}_\CU(G)|$.

If $g$ is an integer-valued flow of a digraph $(G,\rho)$ such that
$0<g(e)<2$ for $e\in E$, then $g\equiv 1$. This means that
$(G,\rho)$ is an edge-disjoint union of directed circuits, i.e., a
directed Eulerian graph. Thus $\varphi_{\Bbb Z}(G,2)=|{\mathcal
O}_\EU(G)|$ by (\ref{Int-Kappa-Decom}), and subsequently,
$\bar\varphi_{\Bbb Z}(G,-2)=(-1)^{n(G)}|{\mathcal O}_\EU(G)|$.
According to Theorem~\ref{Integral-Thm}(d), we have $\kappa_{\Bbb
Z}(G;1,2)=|\mathcal{O}_\EU(G)|$ and $\bar\kappa_{\Bbb
Z}(G;-1,-2)=(-1)^{n(G)}|\mathcal{O}_\EU(G)|$.

Let $(f,g)$ be an integer-valued complementary tension-flow of
$(G,\rho)$ such that $0<f(e)<2$, $0<g(e)<2$ for $e\in E$. Then
$f(e)+g(e)=1$ for all $e\in E$. This means that $(G,\rho)$ is an
edge-disjoint union of a locally directed cut and a directed
Eulerian subgraph. Hence $\kappa_{\Bbb
Z}(G;2,2)=|\mathcal{O}_\CE(G)|$ by Equation~(\ref{Int-Kappa-Decom}).

(c) Note that $\bar\kappa_\rho(G;1,0)=\bar\tau_\rho(G,1)$ and
$\bar\kappa_\rho(G;0,1)=\bar\varphi_\rho(G,1)$. We then have
$\bar\kappa_\rho(G;1,0)=\#[\rho]_\CU$ by Proposition~6.8 in
\cite{Chen-I}, $\bar\kappa_\rho(G;0,1)=\#[\rho]_\EU$ by
Proposition~5.5 in \cite{Chen-II}, and
$\kappa_\rho(G;1,1)=\#[\rho]_\CE$ by Proposition~4.6. Now the
desired formulas follow immediately from
Equation~(\ref{Int-Bar-Kappa-Decom}).

\end{proof}

\section{Modular complementary polynomials}

Let $A,B$ be abelian groups of orders $|A|=p,|B|=q$. Recall the {\em
modular complementary polynomial}
\begin{equation}\label{Gamma-PQRS}
\kappa(G;p,q):=|K(G,\varepsilon;A,B)|.
\end{equation}
To find the relationship between $\kappa$ and $\kappa_{\Bbb Z}$, we
need an equivalence relation on $\mathcal{O}(G)$. Two orientations
$\rho,\sigma$ on $G$ are said to be {\em cut-Eulerian equivalent},
denoted $\rho\sim_\CE\sigma$, if the subgraph induced by the edge
subset
\[
E(\rho\neq\sigma)=\{e\in E: \rho(v,e)\neq\sigma(v,e),\mbox{$v$ is an
endvertex of $e$}\}
\]
is an edge-disjoint union of a locally directed cut and a directed
Eulerian subgraph of both digraphs $(G,\rho)$ and $(G,\sigma)$.

Let $[{\mathcal O}(G)]$ denote the set of cut-Eulerian equivalence
classes of ${\mathcal O}(G)$, and let $\Rep[\mathcal{O}(G)]$ be a
set of distinct representatives of cut-Eulerian equivalence classes
of ${\mathcal O}(G)$. For nonnegative integers $p,q$, recall the
counting function
\begin{align}
\bar\kappa(G;p,q) &=\#\{(\rho,f,g) :
\mbox{$\rho\in\Rep[\mathcal{O}(G)]$, $(f,g)$ is an
integer-valued ten-}\nonumber\\
&\hspace{4ex} \mbox{sion-flow of $(G,\rho)$ s.t. $0\leq f(e)\leq p$,
$0\leq g(e)\leq q$, $e\in E$}\}. \label{Mod-Bar-Kappa-Defn*}
\end{align}
It is clear that $\bar\kappa$ is a polynomial function of
nonnegative integers $p,q$, and is independent of the chosen set
$\mbox{\rm Rep}[\mathcal{O}(G)]$ of distinct representatives, called
the {\em dual modular complementary polynomial} of $G$. We introduce
the topological sum
\[
[\widetilde{\Delta}_\CTF(G,\varepsilon)]
=\sum_{[\rho]\in[\mathcal{O}(G)]}\bar{\Delta}^\rho_\CTF(G,\varepsilon),
\]
defined as a disjoint union of copies of
$\bar\Delta^{\rho}_\CTF(G,\rho)$, one copy for each cut-Eulerian
equivalence class $[\rho]\in[{\mathcal O}(G)]$. Then
$\bar\kappa(G;p,q)$ counts the number of lattice points of
$(p,q)[\widetilde{\Delta}_\CTF(G,\varepsilon)]$.

Let ${\Bbb R}_p={\Bbb R}/p{\Bbb Z}$ and ${\Bbb R}_q={\Bbb R}/q{\Bbb
Z}$ for positive integers $p,q$. There is an obvious homomorphism
$\Mod_{p,q}: ({\Bbb R}\times {\Bbb R})^E \rightarrow ({\Bbb
R}_p\times {\Bbb R}_q)^E$, defined for $(f,g)\in({\Bbb R}\times{\Bbb
R})^E$ by
\begin{equation}\label{Mod-Map}
\Mod_{p,q}(f,g)(e)=(f(e)\: \mod p,\; g(e)\:\mod q),\sp e\in E.
\end{equation}
For two orientations $\rho,\sigma\in\mathcal{O}(G)$ and an edge
subset $S\subseteq E$, there is an involution
$Q^p_{\rho,\sigma,S}:[0,p]^E\rightarrow[0,p]^E$ defined by
\begin{equation}
(Q^{p}_{\rho,\sigma,S}f)(e)=\left\{\begin{array}{rl} p-f(e) &
\mbox{if $e\in S$, $\rho(v,e)\neq\sigma(v,e)$,}\\
f(e) & \mbox{otherwise,}
\end{array}\right.
\end{equation}
where $f\in[0,p]^E$, $e\in E$, and $v$ is an endvertex of $e$. To
find the relationship holding among the polytopes
$p\bar\Delta_\TN(G,\rho)\times q\bar\Delta_\FL(G,\rho)$ with
$\rho\in {\mathcal O}(G)$, we need the involution
\[
Q^{p,q}_{\rho,\sigma}:([0,p]\times [0,q])^E\rightarrow ([0,p]\times
[0,q])^E,
\]
defined for $(f,g)\in([0,p]\times[0,q])^E$ by
\begin{equation}\label{Qpq}
Q^{p,q}_{\rho,\sigma}(f,g) =(Q^{p}_{\rho,\sigma,B_\sigma}f,
Q^{q}_{\rho,\sigma,C_\sigma}g).
\end{equation}
Be care of that the subscript of $B_\sigma$ and $C_\sigma$ in
(\ref{Qpq}) should match the subscript of $Q^{p,q}_{\rho,\sigma}$
and be understood as their edge subsets. We establish the following
lemmas.

\begin{lemma}\label{Cut-Eulerian-BC}
Two orientations $\rho,\sigma\in{\mathcal O}(G)$ are cut-Eulerian
equivalent if and only if
\begin{enumerate}
\item[(a)] $E(B_\rho)=E(B_\sigma)$, $E(C_\rho)=E(C_\sigma)$;

\item[(b)] $E(B_\rho)\cap
E(\rho\neq\sigma)$ is a locally directed cut of $(G,\rho)$;

\item[(c)] $E(C_\rho)\cap E(\rho\neq\sigma)$ is a
directed Eulerian subgraph of $(G,\rho)$.
\end{enumerate}
\end{lemma}
\begin{proof}
The sufficiency is trivial, since $E=E(B_\rho)\sqcup E(C_\rho)$ by
Lemma~\ref{EBC}. For necessity, let us write
$E(\rho\neq\sigma)=B\sqcup C$, where $B$ is a locally directed cut
and $C$ a directed Eulerian subgraph of $(G,\rho)$. Clearly,
$C\subseteq E(C_\rho)$. Since each directed circuit of $C_\rho$ is
edge-disjoint from any directed bond contained in the locally
directed cut $(B,\rho)$ of $(G,\rho)$, then $B\cap
E(C_\rho)=\emptyset$; subsequently, $B\subseteq E(B_\rho)$. Thus we
must have $B=E(B_\rho)\cap E(\rho\neq\sigma)$ and $C=E(C_\rho)\cap
E(\rho\neq\sigma)$. Now it follows that (b) and (c) are
automatically true.

To prove (a), let $\rho'$ be an orientation on $G$ obtained from
$\rho$ by reversing the orientations on the edges of $C$. Then
$B=E(\rho'\neq\sigma)$ and $E-B=E(\rho'=\sigma)$. Since reversing
the orientations on any strongly connected subdigraph does not
change the strong connectedness, we have $E(C_{\rho'})=E(C_\rho)$.
Note that $C_{\rho'}$ and $C_\sigma$ are edge-disjoint from $B$.
This means that $C_{\rho'}$ and $C_\sigma$ are contained in
$E(\rho'=\sigma)$. Hence $C_{\rho'}=C_\sigma$. Therefore
$E(C_\rho)=E(C_\sigma)$, and subsequently, $E(B_\rho)=E(B_\sigma)$.
\end{proof}

\begin{lemma}\label{cut-Euler-P}
{\rm (a)} The relation $\sim_{\CE}$ is indeed an equivalence
relation on ${\mathcal O}(G)$.

{\rm (b)} Let $\rho,\sigma\in{\mathcal O}(G)$ and
$\rho\sim_\CE\sigma$. If $B_\rho$ is a locally directed cut, so is
$B_\sigma$. If $C_\rho$ is a directed Eulerian graph, so is
$C_\sigma$.

{\rm (c)} Let $\rho,\sigma\in{\mathcal O}(G)$ and
$\rho\sim_\CE\sigma$. The restriction
\[
\hspace{1.5cm}
Q_{\rho,\sigma}^{p,q}:p\bar\Delta_\TN(G,B_\sigma)\times
q\bar\Delta_\FL(G,C_\sigma)\rightarrow
p\bar\Delta_\TN(G,B_\rho)\times q\bar\Delta_\FL(G,C_\rho)
\]
is a bijection, sending lattice points to lattice points
bijectively.

{\rm (d)} $\kappa_\rho(G;x,y) =\kappa_\sigma(G;x,y)$ if
$\rho\sim_\CE\sigma$.
\end{lemma}
\begin{proof}
(a) Let $\varepsilon_i\in\mathcal{O}(G)$ ($i=1,2,3$) be orientations
such that $\varepsilon_1\sim_\CE\varepsilon_2$ and
$\varepsilon_2\sim_\CE\varepsilon_3$. Set $B:=E(B_{\varepsilon_i})$
and $C:=E(C_{\varepsilon_i})$ by Lemma~\ref{Cut-Eulerian-BC}. Then
$\varepsilon_1\sim_\CU\varepsilon_2$,
$\varepsilon_2\sim_\CU\varepsilon_3$ on $B$, and
$\varepsilon_1\sim_\EU\varepsilon_2$,
$\varepsilon_2\sim_\EU\varepsilon_3$ on $C$. Thus
$\varepsilon_1\sim_\CU\varepsilon_3$ on $B$ by Lemma~6.5(a) of
\cite{Chen-I}, and $\varepsilon_1\sim_\CU\varepsilon_3$ on $C$ by
Lemma~5.1(a) of \cite{Chen-II}. Hence
$\varepsilon_1\sim_\CE\varepsilon_3$.

(b) It follows from the fact that the subdigraph
$B_\rho|E(\rho=\sigma)$ is a directed cut and that
$C_\rho|E(\rho=\sigma)$ is a directed Eulerian subgraph of both
$(G,\rho)$ and $(G,\sigma)$.

(c) Let $E(\rho\neq\sigma)=B'\sqcup C'$, where $B'$ is a locally
directed cut and $C'$ a directed Eulerian subgraph with both
orientations $\rho$ and $\sigma$. Let $f\in
p\bar\Delta_\TN(G,\sigma)$ and $g\in q\bar\Delta_\FL(G,\sigma)$. We
need to show that $Q^p_{\rho,\sigma,B_\sigma}f\in
p\bar\Delta_\TN(G,\rho)$ and $Q^q_{\rho,\sigma,C_\sigma}g\in
q\bar\Delta_\FL(G,\rho)$. In fact, given an arbitrary directed
circuit $(C,\varepsilon_\CU)$ and directed bond $(B,\varepsilon_B)$
of $(G,\sigma)$. We have
\[
\sum_{e\in C}[\sigma,\varepsilon_\CU](e)f(e)=0, \quad \sum_{e\in
B}[\sigma,\varepsilon_B](e)g(e)=0.
\]
Then $I:=\sum_{e\in C}[\rho,\varepsilon_\CU](e)
(Q^p_{\rho,\sigma,B_\sigma}f)(e)$ can be written as
\begin{eqnarray*}
I &=& \sum_{e\in C\cap(C'\cup
E(\rho=\sigma))}[\rho,\varepsilon_\CU](e)f(e) +\sum_{e\in
C\cap B'}[\rho,\varepsilon_\CU](e)(p-f(e))\\
&=& \sum_{e\in C}[\sigma,\varepsilon_\CU](e)f(e) -p\sum_{e\in C\cap
B'}[\sigma,\varepsilon_\CU](e)=0.
\end{eqnarray*}
The second sum equal to zero follows from the fact that
$(B',\sigma)$ is a locally directed cut of $(G,\sigma)$. Hence
$Q^p_{\rho,\sigma,B_\sigma}f\in p\bar\Delta_\TN(G,\rho)$.
Analogously,
\[
J:=\sum_{e\in
B}[\rho,\varepsilon_B](e)(Q^q_{\rho,\sigma,C_\sigma}g)(e)
\]
can be written as
\begin{eqnarray*}
J &=& \sum_{e\in B\cap(B'\cup
E(\rho=\sigma))}[\rho,\varepsilon_B](e)g(e) +\sum_{e\in
B\cap C'}[\rho,\varepsilon_B](e)(q-g(e))\\
&=& \sum_{e\in B}[\sigma,\varepsilon_B](e)g(e) -q\sum_{e\in B\cap
C'}[\sigma,\varepsilon_B](e)=0.
\end{eqnarray*}
The second sum equal to zero follows from the fact that
$(C',\sigma)$ is a directed Eulerian subgraph of $(G,\sigma)$.
Therefore $Q^q_{\rho,\sigma,C_\sigma}f\in q\bar\Delta_\FL(G,\rho)$.
It is clear that $Q_{\rho,\sigma}^{p,q} =Q_{\rho,\sigma,B_\sigma}^p
\times Q_{\rho,\sigma,C_\sigma}^q$ and sends lattice points
bijectively.

(d) It is a consequence of (c).
\end{proof}

For a real-valued function $f:E\rightarrow\Bbb R$ and an orientation
$\rho$ on $G$, we associate an orientation $\rho_f$ on $G$, defined
for each $(v,e)\in V\times E$ by
\begin{equation}\label{Orientation-Function}
\rho_f(v,e)=\left\{\begin{array}{rl} \rho(v,e)
&\mbox{if $f(e)>0$,}\\
-\rho(v,e) &\mbox{if $f(e)\leq 0$.}
\end{array}\right.
\end{equation}
Conversely, for orientations $\rho,\sigma\in{\mathcal O}(G)$, we
associate a {\em symmetric difference function}
$I_{\rho,\sigma}:E\rightarrow\{0,1\}$, defined for each edge $e\in
E$ (at its one endvertex $v$) by
\begin{equation}
I_{\rho,\sigma}(e)=\left\{\begin{array}{rl} 1
&\mbox{if $\rho(v,e)\neq \sigma(v,e)$,}\\
0 & \mbox{if $\rho(v,e)=\sigma(v,e)$}.
\end{array}\right. \label{Indicator-Orientation}
\end{equation}

\begin{lemma}
Let $(f_i,g_i)\in p\Delta_\TN(G,B_\varepsilon)\times
q\Delta_\FL(G,C_\varepsilon)$, $i=1,2$. If
\[
(f_1,g_1)\equiv (f_2,g_2)~\mod (p, q),
\]
then $\varepsilon_{f_1+g_1}$ and $\varepsilon_{f_2+g_2}$ are
cut-Eulerian equivalent.
\end{lemma}
\begin{proof}
Note that $f_i(e)=0$ for $e\in E(C_\varepsilon)$, and that
$g_i(e)=0$ for $e\in E(B_\varepsilon)$. Then
$\varepsilon_{f_i}=\varepsilon_{f_i+g_i}$ on $E(B_\varepsilon)$, and
$\varepsilon_{g_i}=\varepsilon_{f_i+g_i}$ on $E(C_\varepsilon)$.
Since $f_1\equiv f_2\;\mod p$, then $\varepsilon_{f_1}$ and
$\varepsilon_{f_2}$ are cut equivalent by Lemma~6.6 of
\cite{Chen-I}. Analogously, since $g_1\equiv g_2\;\mod q$, then
$\varepsilon_{g_1}$ and $\varepsilon_{g_2}$ are Eulerian equivalent
by Lemma~5.3 of \cite{Chen-II}. Since
$E(\varepsilon_{f_1}\neq\varepsilon_{f_2})\subseteq
E(B_\varepsilon)$ and
$E(\varepsilon_{g_1}\neq\varepsilon_{g_2})\subseteq
E(C_\varepsilon)$, it follows that the edge set
$D:=E(\varepsilon_{f_1+g_1}\neq\varepsilon_{f_2+g_2})$ can be
written as
\begin{eqnarray*}
D &=& (D\cap B_\varepsilon)\sqcup (D\cap C_\varepsilon)\\
&=& (E(\varepsilon_{f_1}\neq \varepsilon_{f_2})\cap
E(B_\varepsilon))\sqcup (E(\varepsilon_{g_1}\neq
\varepsilon_{g_2})\cap E(C_\varepsilon)) \\
&=& E(\varepsilon_{f_1}\neq \varepsilon_{f_2})\sqcup
E(\varepsilon_{g_1}\neq \varepsilon_{g_2}),
\end{eqnarray*}
which is an edge-disjoint union of a locally directed cut and a
directed Eulerian subgraph with orientations
$\varepsilon_{f_i+g_i}$. Hence $\varepsilon_{f_1+g_1}$ and
$\varepsilon_{f_2+g_2}$ are cut-Eulerian equivalent.
\end{proof}

\begin{lemma}\label{Surjective}
Let $K_{\Bbb
Z}(G,\varepsilon;p,q)=(p,q)\Delta_\CTF(G,\varepsilon)\cap({\Bbb
Z}\times{\Bbb Z})^E$. The restriction
\[
\Mod_{p,q}: K_{\Bbb Z}(G,\varepsilon;p,q) \rightarrow
K(G,\varepsilon;{\Bbb Z}_p, {\Bbb Z}_q)
\]
is well-defined and surjective.
\end{lemma}
\begin{proof}
It is clear that $\Mod_{p,q}$ is well-defined. Let $(\tilde f,\tilde
g)\in K(G,\varepsilon;{\Bbb Z}_p, {\Bbb Z}_q)$, i.e., $\tilde f\in
T(G,\varepsilon;{\Bbb Z}_p)$, $\tilde g\in F(G,\varepsilon;{\Bbb
Z}_q)$, and $\supp\tilde g=\ker\tilde f$. Then by Lemma~6.1 of
\cite{Chen-I}, there exists an integer-valued $p$-tension $f$ of
$(G,\varepsilon)$ such that $\Mod_pf=\tilde f$, and by Lemma~5.3 of
\cite{Chen-II}, there exists an integer-valued $q$-flow $g$ of
$(G,\varepsilon)$ such that $\Mod_pg=\tilde g$. Moreover, $f(e)=0$
if and only if $\tilde f(e)=0$; $g(e)=0$ if and only if $\tilde
g(e)=0$. Thus $\ker f=\ker\tilde f$, $\supp g=\supp\tilde g$.
Therefore $\supp g=\ker f$. This means that $(f,g)\in K_{{\Bbb
Z}}(G,\varepsilon;p,q)$ and $\Mod_{p.q}(f,g)=(\tilde f,\tilde g)$.
We have proved the surjectivity of the map $\Mod_{p,q}$.
\end{proof}

\begin{lemma}\label{OPQ}
Let $\rho,\sigma,\omega\in{\mathcal O}(G)$ be cut-Eulerian
equivalent orientations, and let $(f,g)\in
(p,q)\Delta^\rho_\CTF(G,\varepsilon)$. Then
\begin{enumerate}
\item[(a)] $\varepsilon_{f+g}=\rho$, $f|_{C_\rho}=0$, $g|_{B_\rho}=0$.

\item[(b)] $P_{\varepsilon,\sigma}Q^{p,q}_{\sigma,\rho}
P_{\rho,\varepsilon}((p,q)\Delta^\rho_\CTF(G,\varepsilon))
 =(p,q)\Delta^\sigma_\CTF(G,\varepsilon)$.

\item[(c)] $P_{\varepsilon,\sigma}Q^{p,q}_{\sigma,\rho}
P_{\rho,\varepsilon}(f,g) =
P_{\varepsilon,\omega}Q^{p,q}_{\omega,\rho}
P_{\rho,\varepsilon}(f,g)$ if and only if $\sigma=\omega$.

\item[(d)] $K(G,\varepsilon;p,q)\cap\Mod_{p,q}^{-1}(\Mod_{p,q}(f,g))
=\{P_{\varepsilon,\alpha}Q^{p,q}_{\alpha,\rho}
P_{\rho,\varepsilon}(f,g):\alpha\sim_\CE\rho\}$.
\end{enumerate}
\end{lemma}
\begin{proof}
(a) Recall that $(f,g)\in\Delta^\rho_\CTF(G,\varepsilon)$ is
equivalent to $[\rho,\varepsilon](e)(f+g)(e)>0$ for all $e\in E$.
Then $(f+g)(e)>0$ if and only if $\rho(v,e)=\varepsilon(v,e)$ for an
endvertex $v$ of $e$. Since
$\varepsilon_{f+g}(v,e)=\varepsilon(v,e)$ if and only if
$(f+g)(e)>0$, we see that $\varepsilon_{f+g}=\rho$.

Note that $[\rho,\varepsilon](f,g)\in\Delta_\CTF^+(G,\rho)$ by
Lemma~\ref{Delta-CTF-Decomposition}(b). Then $f$ is nowhere zero on
$B_\rho$ and vanishes on $C_\rho$; and $g$ is nowhere zero on
$C_\rho$ and vanishes on $B_\rho$.

(b) Since $P_{\rho,\varepsilon}$, $P_{\varepsilon,\sigma}$, and
$Q^{p,q}_{\sigma,\rho}$ are involutions, it follows that the
composition $P_{\varepsilon,\rho} Q^{p,q}_{\rho,\sigma,S}
P_{\sigma,\varepsilon}$ is an involution. Note that
\[
P_{\rho,\varepsilon}((p,q)\Delta^\rho_\CTF(G,\varepsilon))
=(p,q)\Delta^+_\CTF(G,\rho),
\]
\[
Q^{p,q}_{\sigma,\rho}((p,q)\Delta^+_\CTF(G,\rho))
=(p,q)\Delta^+_\CTF(G,\sigma),
\]
and $P_{\sigma,\varepsilon}=P_{\varepsilon,\sigma}$. The desired
identity follows.

(c) The equation $P_{\varepsilon,\sigma}Q^{p,q}_{\sigma,\rho}
P_{\rho,\varepsilon}(f,g) =
P_{\varepsilon,\omega}Q^{p,q}_{\omega,\rho}
P_{\rho,\varepsilon}(f,g)$ is equivalent to
\begin{equation}\label{Pf}
P_{\varepsilon,\sigma}Q^{p}_{\sigma,\rho,B_\rho}
P_{\rho,\varepsilon} f =
P_{\varepsilon,\omega}Q^{p}_{\omega,\rho,B_\rho}
P_{\rho,\varepsilon} f,
\end{equation}
\begin{equation}\label{Pg}
P_{\varepsilon,\sigma}Q^{q}_{\sigma,\rho,C_\rho}
P_{\rho,\varepsilon} g =
P_{\varepsilon,\omega}Q^{q}_{\omega,\rho,C_\rho}
P_{\rho,\varepsilon} g.
\end{equation}
It follows that (\ref{Pf}) is equivalent to $\sigma=\omega$ on
$B_\rho$ by Lemma~6.7(c) of \cite{Chen-I}, and that (\ref{Pg}) is
equivalent to $\sigma=\omega$ on $C_\rho$ by Lemma~5.4(c) of
\cite{Chen-II}. Hence (\ref{Pf}) and (\ref{Pg}) are equivalent to
$\sigma=\omega$ on $E$.

(d) Let $T(G,B_\rho,\varepsilon;p)$ be the set of $p$-tensions of
$(G,\varepsilon)$ vanishing on $C_\rho$, and
$F(G,C_\rho,\varepsilon;q)$ the set of $q$-flows of
$(G,\varepsilon)$ vanishing on $B_\rho$. Then
$T(G,B_\rho,\varepsilon;p)$ can be identified to
$T(G/C_\rho,\varepsilon;p)$, and $F(G,C_\rho,\varepsilon;q)$ to
$F(G|C_\rho,\varepsilon;q)$. By Lemma~6.7(d) of \cite{Chen-I},
$T(G,B_\varepsilon;p)\cap \Mod_p^{-1}\bigl(\Mod_p f\bigr)$ consists
of tensions $P_{\varepsilon,\alpha} Q^p_{\alpha,\rho,B_\rho}
P_{\rho,\varepsilon} f$, where $\alpha\sim_{\CU}\rho$ on $B_\rho$.
Likewise, by Lemma~5.4(d) of \cite{Chen-II},
$F(G,C_\varepsilon;q)\cap \Mod_q^{-1}\bigl(\Mod_q g\bigr)$ consists
of flows $P_{\varepsilon,\alpha} Q^q_{\alpha,\rho,C_\rho}
P_{\rho,\varepsilon}g$, where $\alpha\sim_{\EU}\rho$ on $C_\rho$.
Clearly, $P_{\varepsilon,\alpha} Q^p_{\alpha,\rho,B_\rho}
P_{\rho,\varepsilon} f$ vanishes on $C_\rho$, and
$P_{\varepsilon,\alpha} Q^q_{\alpha,\rho,C_\rho}
P_{\rho,\varepsilon}g$ vanishes on $B_\rho$. It follows that
$K(G,\varepsilon;p,q)\cap\Mod_{p,q}^{-1}(\Mod_{p,q}(f,g))$ consists
of tension-flows $P_{\varepsilon,\alpha}Q^{p,q}_{\alpha,\rho}
P_{\rho,\varepsilon}(f,g)$, where $\alpha\sim_{\CE}\rho$ on $E$.
\end{proof}

\begin{prop}\label{equivalence-card}
For each orientation $\rho$, let $[\rho]_\CU$, $[\rho]_\EU$, and
$[\rho]_\CE$ denote its equivalence classes under the cut
equivalence, Eulerian equivalence, and cut-Eulerian equivalence
relations on ${\mathcal O}(G)$ respectively. Then
\begin{equation}\label{CE=CUEU}
\#[\rho]_\CE=\#[\rho]_\CU\cdot \#[\rho]_\EU
\end{equation}
and equals the number of $0$-$1$ complementary tension-flows of
$(G,\rho)$, i.e.,
\begin{equation}\label{CE-Delta}
\#[\rho]_\CE= \bar\kappa_\rho(G;1,1)
=\left|\bar\Delta^+_\CTF(G,\rho)\cap({\Bbb Z}\times{\Bbb
Z})^E\right|.
\end{equation}
\end{prop}
\begin{proof}
The map $[\rho]_\CU\times [\rho]_\EU\rightarrow [\rho]_\CE$ given by
$(\sigma,\omega)\mapsto (\sigma{|B_\rho})\cup(\omega{|C_\rho})$ is
clearly a bijection, where $\sigma{|B_\rho}$ is the restriction of
the orientation $\sigma$ to  $E(B_\rho)$, and $\omega{|C_\rho}$ the
restriction of $\omega$ to $E(C_\rho)$. The equality (\ref{CE=CUEU})
follows immediately.

Let $\sigma$ be an orientation that is cut-Eulerian equivalent to
$\rho$. Let $I_\sigma$ denote the symmetric difference function
$I_{\sigma,\rho}$ defined by (\ref{Indicator-Orientation}). Then
$I_{\sigma}|_{B_\rho}$ is a tension on $(G/C_\rho,\rho)$ by
Proposition~6.8 in \cite{Chen-I}, and can be extended to a tension
$I^\prime_{\sigma}$ on $(G,\rho)$ such that
$I^\prime_{\sigma}|_{C_\rho}=0$. Likewise, $I_{\sigma}|_{C_\rho}$ is
a flow on $(G|C_\rho,\rho)$ by Lemma~5.5 in \cite{Chen-II}, and can
be extended to a flow $I^{\prime\prime}_{\sigma}$ on $(G,\rho)$ such
that $I^{\prime\prime}_{\sigma}|_{B_\rho}=0$. Hence $I_{\sigma}$ is
decomposed into a 0-1 tension-flow
$(I^\prime_\sigma,I^{\prime\prime}_\sigma)$ of $(G,\rho)$. We then
have a well-defined map
\[
[\rho]_\CE =\{\sigma\in\mathcal{O}(G):\sigma\sim_{\CE}\rho\}
\rightarrow \bar\Delta^+_\CTF(G,\rho)\cap ({\Bbb Z}\times{\Bbb
Z})^E, \sp \sigma\mapsto (I^\prime_\sigma,I^{\prime\prime}_\sigma).
\]
The map is clearly injective.

For surjectivity of the map, let $(f,g)\in\bar\Delta^+_\CTF(G,\rho)$
be a 0-1 tension-flow. Set $\sigma:=-\rho_{f+g}$, where $\rho_{f+g}$
is the orientation associated with $f+g$ and $\rho$ by
(\ref{Orientation-Function}). Then $I_{\sigma,\rho}=f+g$. Since
\[
\begin{split}
E(\sigma\neq\rho) &=\{e\in E: (f+g)(e)=1\}
\\ &=\{e\in E(B_\rho):f(e)=1\} \sqcup \{e\in
E(C_\rho):g(e)=1\}
\end{split}
\]
is a disjoint union of a locally directed cut and a directed
Eulerian subgraph with the orientation $\rho$, we see that
$\sigma\sim_{\CE}\rho$. Thus the map is surjective.
\end{proof}

\begin{lemma}
Let $[\rho]=\{\sigma\in\mathcal{O}(G):\sigma\sim_\CE\rho\}$ and
$\Delta^{[\rho]}_\CTF(G,\varepsilon) =\bigsqcup_{\sigma\in[\rho]}
\Delta^\sigma_\CTF(G,\varepsilon)$ for each $\rho\in\mathcal{O}(G)$.
Then
\begin{align}
\kappa_\rho(G;p,q) &=\#\Mod_{p,q}
\big((p,q)\Delta^\rho_\CTF(G,\varepsilon) \cap ({\Bbb Z}\times{\Bbb
Z})^E\big)
\label{Kappa-Z-rho}\\
&= \#\Mod_{p,q} \big((p,q)\Delta^{[\rho]}_\CTF(G,\varepsilon)
\cap({\Bbb Z}\times{\Bbb Z})^E\big); \label{Kappa-Z-[rho]}
\end{align}
\begin{align}
K(G,\varepsilon;{\Bbb Z}_p,{\Bbb Z}_q)
&=\bigsqcup_{[\rho]\in[{\mathcal O}(G)]} \Mod_{p,q}
\big((p,q)\Delta^\rho_\CTF(G,\varepsilon)\cap ({\Bbb Z}\times{\Bbb
Z})^E\big)
\label{Modular-Decom-rho}\\
&= \bigsqcup_{[\rho]\in[{\mathcal O}(G)]} \Mod_{p,q}
\big((p,q)\Delta^{[\rho]}_\CTF(G,\varepsilon)\cap ({\Bbb
Z}\times{\Bbb Z})^E\big). \label{Modular-Decom-[rho]}
\end{align}
\end{lemma}
\begin{proof}
Notice that for each tension-flow $(f,g)\in
(p,q)\Delta^\rho_\CTF(G,\varepsilon)$, we have
\begin{align}
\#[\rho] &=\#\{P_{\varepsilon,\sigma} Q^{p,q}_{\sigma,\rho}
P_{\rho,\varepsilon} (f,g) : \sigma\sim_\CE\rho \} \nonumber\\
&=\left|K(G,\varepsilon;p,q)\cap\Mod_{p,q}^{-1}\Mod_{p,q}(f,g)\right|;
\label{MM}
\end{align}
the first equality follows from Lemma~\ref{OPQ}(c) and the second
from Lemma~\ref{OPQ}(d). Set
$D:=(p,q)\Delta^{[\rho]}_\CTF(G,\varepsilon)$. Then
\begin{align}
D &= \bigsqcup_{\sigma\in[\rho]}
P_{\varepsilon,\sigma}Q^{p,q}_{\sigma,\rho} P_{\rho,\varepsilon}
\big((p,q)\Delta^\rho_\CTF(G,\varepsilon)\big) \nonumber\\
&= \bigsqcup_{\sigma\in[\rho]\atop
(f,g)\in(p,q)\Delta^\rho_\CTF(G,\varepsilon)} \big\{
P_{\varepsilon,\sigma} Q^{p,q}_{\sigma,\rho}
P_{\rho,\varepsilon}(f,g)\big\} \nonumber\\
&= \bigsqcup_{(f,g)\in(p,q) \Delta^\rho_\CTF(G,\varepsilon)}
K(G,\varepsilon;p,q)\cap \Mod_{p,q}^{-1}
\Mod_{p,q}(f,g) \nonumber\\
&= K(G,\varepsilon;p,q)\cap \Mod_{p,q}^{-1}
\Mod_{p,q}(p,q)\Delta^\rho_\CTF(G,\varepsilon); \label{Delta[rho]}
\end{align}
the first equality follows from Lemma~\ref{OPQ}(b), the second one
is straightforward, the third one follows from Lemma~\ref{OPQ}(d),
and the last one is trivial. Since the orientation $\rho$ on the
left-hand side of (\ref{Delta[rho]}) can be replaced by any
orientation $\sigma\in[\rho]$, we thus have
\begin{equation}\label{Tilde-D}
D = K(G,\varepsilon;p,q)\cap \Mod_{p,q}^{-1}
\Mod_{p,q}(p,q)\Delta^{[\rho]}_\CTF(G,\varepsilon).
\end{equation}
Now on the one hand, by (\ref{MM}), (\ref{Delta[rho]}), and
(\ref{Tilde-D}) we have
\begin{align}
\big|D\cap ({\Bbb Z}\times{\Bbb Z})^E\big| &= \big|\Mod_{p,q}
\big((p,q)\Delta^\rho_\CTF(G,\varepsilon)\cap ({\Bbb Z}\times{\Bbb
Z})^E\big)\big|\cdot
\#[\rho] \label{Delta-Varrho}\\
&= \big|\Mod_{p,q} \big((p,q)\Delta^{[\rho]}_\CTF(G,\varepsilon)\cap
({\Bbb Z}\times{\Bbb Z})^E\big)\big|\cdot\#[\rho].
\label{Delta-[Varrho]}
\end{align}
On the other hand, recall that $\kappa_\rho(G;p,q)
=\kappa_\sigma(G;p,q)$ if $\sigma\sim_\CE\rho$ (see
Lemma~\ref{cut-Euler-P}(d)); by definition of
$\Delta^{[\rho]}_\CTF(G,\varepsilon)$ we trivially have
\begin{align}
\big|D\cap ({\Bbb Z}\times{\Bbb Z})^E\big|
&=\kappa_\rho(G;p,q)\cdot\#[\rho] \label{Delta[Varrho]}.
\end{align}
Equate (\ref{Delta-Varrho}), (\ref{Delta-[Varrho]}), and
(\ref{Delta[Varrho]}); we obtain (\ref{Kappa-Z-rho}) and
(\ref{Kappa-Z-[rho]}).

Notice the disjoint decomposition in
Lemma~\ref{Delta-CTF-Decomposition}(a). Using the notation
$\Delta^{[\rho]}_\CTF(G,\varepsilon)$ and applying
(\ref{Delta[rho]}), we have
\begin{align*}
(p,q)\Delta_\CTF(G,\varepsilon) &= \bigsqcup_{[\rho]\in[{\mathcal
O}(G)]} \bigsqcup_{\sigma\in[\rho]}
(p,q)\Delta^\sigma_\CTF(G,\varepsilon) \\
&= \bigsqcup_{[\rho]\in[{\mathcal O}(G)]}
K(G,\varepsilon;p,q)\cap
\Mod_{p,q}^{-1}\Mod_{p,q}(p,q)\Delta^\rho_\CTF(G,\varepsilon)\\
& = \bigsqcup_{[\rho]\in[{\mathcal O}(G)]} K(G,\varepsilon;p,q)\cap
\Mod_{p,q}^{-1}\Mod_{p,q}(p,q)\Delta^{[\rho]}_\CTF(G,\varepsilon).
\end{align*}
Note that $K(G,\varepsilon;p,q)=(p,q)\Delta_\CTF(G,\varepsilon)$.
Apply the map $\Mod_{p,q}$ to both sides and restrict to the
integral lattice $({\Bbb Z}\times{\Bbb Z})^{E}$; we obtain
(\ref{Modular-Decom-rho}) and (\ref{Modular-Decom-[rho]}) by
Lemma~\ref{Surjective}.
\end{proof}

{\sc Proof of Theorem~\ref{Modular-Thm}} \vspace{1ex}

(a) The polynomiality follows from Decomposition Formulas
(\ref{Mod-Kappa-Decom}) and (\ref{Mod-Bar-Kappa-Decom}). The
independence of the chosen $\Rep[\mathcal{O}(G)]$ follows from the
fact that $\kappa_\rho(G;x,y) =\kappa_\sigma(G;x,y)$ if
$\rho\sim_\CE\sigma$.

(b) Let $p,q$ be positive integers. Taking the cardinality on both
sides of
(\ref{Modular-Decom-[rho]}), we have
\[
\kappa(G;p,q)=\sum_{[\rho]\in[\mathcal{O}(G)]} \big|\Mod_{p,q}
((p,q)\Delta^{[\rho]}_\CTF(G,\varepsilon)\cap ({\Bbb Z}\times{\Bbb
Z})^E)\big|.
\]
Taking account of (\ref{Kappa-Z-[rho]}), we obtain
(\ref{Mod-Kappa-Decom}). Equation (\ref{Mod-Bar-Kappa-Decom})
follows by definition of $\bar\kappa$.

(c) The Reciprocity Laws are trivial by the Reciprocity Law
(\ref{Local-Reciprocity-Law}).

(d) Let $A,B$ be finite abelian groups. If $|A|=p,|B|=1$, i.e.,
$B=\{0\}$, then a tension-flow $(f,g)\in\Omega(G,\varepsilon;A,B)$
is complementary if and only if $g\equiv 0$ and $f$ is a
nowhere-zero tension of $(G,\varepsilon)$. Thus $\kappa(G;
x,1)=\tau(G,x)$. If $|A|=1$, $|B|=q$, i.e., $A=\{0\}$, then a
tension-flow $(f,g)\in\Omega(G,\varepsilon;A,B)$ is complementary if
and only if $f\equiv 0$ and $g$ is a nowhere-zero flow of
$(G,\varepsilon)$. Thus $\kappa(G; 1,y)=\varphi(G,y)$. We have
proved (\ref{Mod-Kappa-Special-Variable}).

Since $\kappa_\rho(G; x,1)=\tau_\rho(G,x)$ and
$\kappa_\rho(G;1,y)=\varphi_\rho(G,y)$, the decomposition formulas
(\ref{Mod-Kappa-Decom}) and (\ref{Mod-Bar-Kappa-Decom}) become the
following
\begin{align}
\tau(G,x)&=\sum_{[\rho]\in[\mathcal{O}_\AC(G)]}
\tau_\rho(G,x), \label{Tau-Decom}\\
\varphi(G,y)&=\sum_{[\rho]\in[\mathcal{O}_\TC(G)]}
\varphi_\rho(G,y). \label{Phi-Decom}
\end{align}
The dual polynomials $\bar\tau$ and $\bar\varphi$ have the similar
decomposition formulas (by definitions):
\begin{align}
\bar\tau(G,x)&=\sum_{[\rho]\in[\mathcal{O}_\AC(G)]}
\bar\tau_\rho(G,x), \label{Bar-Tau-Decom}\\
\bar\varphi(G,y)&=\sum_{[\rho]\in[\mathcal{O}_\TC(G)]}
\bar\varphi_\rho(G,y). \label{Bar-Phi-Decom}
\end{align}
Consider the case of $y=-1$ and the case of $x=-1$; we obtain
(\ref{Mod-Bar-Kappa-Special-Variable}) as follows:
\begin{align*}
\bar\kappa(G;x,-1) &= \sum_{[\rho]\in[\mathcal{O}_\AC(G)]}
(-1)^{r(G)+|E(C_\rho)|}\kappa_\rho(G;-x,1)\\
&= \sum_{[\rho]\in[\mathcal{O}_\AC(G)]}
(-1)^{r(G)}\tau_\rho(G,-x) \sp (\mbox{since $|C_\rho|=0$})\\
&= \sum_{[\rho]\in[\mathcal{O}_\AC(G)]} \bar\tau_\rho(G,x)
=\bar\tau(G,x),
\end{align*}
\begin{align*}
\bar\kappa(G;-1,y) &= \sum_{[\rho]\in[\mathcal{O}_\TC(G)]}
(-1)^{n(G)+|E(B_\rho)|}\kappa_\rho(G;1,-y)\\
&= \sum_{[\rho]\in[\mathcal{O}_\TC(G)]}
(-1)^{n(G)}\varphi_\rho(G,-y) \sp (\mbox{since $|B_\rho|=0$})\\
&= \sum_{[\rho]\in[\mathcal{O}_\TC(G)]} \bar\varphi_\rho(G,y)
=\bar\varphi(G,y),
\end{align*}
where the first equality of both idenitites follows from
(\ref{Mod-Bar-Kappa-Reciprocity}), the second equality of both
follows from (\ref{Mod-Kappa-Special-Variable}), the third equality
of both follows from (\ref{Local-Reciprocity-Law}) (with
$E(C_\rho)=\emptyset$ and $E(B_\rho)=\emptyset$ respectively), the
last equality in the former follows from (\ref{Bar-Tau-Decom}), and
the last equality of the latter follows from (\ref{Bar-Phi-Decom}).

(e) For a subset $X\subseteq E$ and an orientation
$\rho\in\mathcal{O}(G)$, we identify the edge set $E(G/X)$ as $E-X$,
write the restriction of $\rho$ to $X$ as $\rho_{|X}$, and the
induced orientation by $\rho$ on $G/X$ as $\rho_{/X}$. If
$X=E(C_\rho)$, then $\rho_{/X}$ is acyclic and $\rho_{|X}$ is
totally cyclic. Thus
\begin{align*}
\kappa(G;x,y)&=\sum_{X\subseteq E}
\sum_{\rho\in[\mathcal{O}(G)]\atop E(C_\rho)=X}
\tau_{\rho_{/X}}(G/X,x)\,\varphi_{\rho_{|X}}(G|X,y) \\
&=\sum_{X\subseteq E} \sum_{\rho\in[\mathcal{O}_\AC(G/X)]\atop
\sigma\in[\mathcal{O}_\TC(G|X)]}
\tau_\rho(G/X,x)\,\varphi_\sigma(G|X,y) \\
&=\sum_{X\subseteq E} \tau(G/X,x)\,\varphi(G|X,y).
\end{align*}
The last equality follows from the decomposition (\ref{Tau-Decom})
to graph $G/X$ and the decomposition (\ref{Phi-Decom}) to graph
$G|X$. Analogously,
\begin{align*}
\bar\kappa(G;x,y)&=\sum_{X\subseteq E}
\sum_{\rho\in[\mathcal{O}(G)]\atop E(C_\rho)=X}
\bar\tau_{\rho_{/X}}(G/X,x)\,\bar\varphi_{\rho_{|X}}(G|X,y) \\
&=\sum_{X\subseteq E} \sum_{\rho\in[\mathcal{O}_\AC(G/X)]\atop
\sigma\in[\mathcal{O}_\TC(G|X)]}
\bar\tau_\rho(G/X,x)\,\bar\varphi_\sigma(G|X,y) \\
&=\sum_{X\subseteq E} \bar\tau(G/X,x)\,\bar\varphi(G|X,y).
\end{align*}
The last equality follows from the decomposition
(\ref{Bar-Tau-Decom}) to graph $G/X$ and the decomposition
(\ref{Bar-Phi-Decom}) to graph $G|X$. \hfill{$\Box$} \vspace{1ex}

There are specific combinatorial interpretations on $\kappa$ and
$\bar\kappa$ at some special integers, similar to that of
Corollary~\ref{Integral-Kappa-Special}. The special values of the
Tutte polynomial $T_G$ in the following corollary are observed
directly by Gioan \cite{Gioan}  by cycle-cocycle reversing systems.
Recall that $\mathcal{O}_\CU(G)$ ($\mathcal{O}_\EU(G)$) is the set
of orientations $\rho$ on $G$ such that $(G,\rho)$ is a locally
directed cut (directed Eulerian subgraph), and that
$\mathcal{O}_\CE(G)$ is the set of orientations $\rho$ on $G$ such
that $(G,\rho)$ is an edge-disjoint union of a locally directed cut
and a directed Eulerian subgraph. Note that
\begin{equation}
\mathcal{O}_\CU(G)\subseteq\mathcal{O}_\AC(G),\quad
\mathcal{O}_\EU(G)\subseteq\mathcal{O}_\TC(G).
\end{equation}
Moreover, if $\rho\in \mathcal{O}_\CU(G)$ $(\mathcal{O}_\EU(G))$ and
$\rho\sim_\CE\sigma$, then $\sigma\in \mathcal{O}_\CU(G)$
$(\mathcal{O}_\EU(G))$. Subsequently, if $\rho\in
\mathcal{O}_\CE(G)$ and $\rho\sim_\CE\sigma$, then $\sigma\in
\mathcal{O}_\CE(G)$. We have the following corollary.

\begin{cor}\label{Modular-Kappa-Special}
Let $[\mathcal{O}(G)]$, $[\mathcal{O}_\AC(G)]$,
$[\mathcal{O}_\TC(G)]$, $[\mathcal{O}_\CU(G)]$,
$[\mathcal{O}_\EU(G)]$, $[\mathcal{O}_\CE(G)]$ denote the sets of
cut-Eulerian equivalence classes of $\mathcal{O}(G)$,
$\mathcal{O}_\AC(G)$, $\mathcal{O}_\TC(G)$, $\mathcal{O}_\CU(G)$,
$\mathcal{O}_\EU(G)$, $\mathcal{O}_\CE(G)$ respectively. Then
\begin{align*}
T_G(0,0) &= \bar\kappa(G;-1,-1) = \kappa(G;1,1)= 0,\\
T_G(1,1) &= \bar\kappa(G;0,0)  = \#[\mathcal{O}(G)],\\
T_G(2,2) &= \bar\kappa(G;1,1)  = \#{\mathcal O}(G);
\end{align*}
\begin{align*}
|T_G(0,-1)| &= |\bar\kappa(G;-1,-2)| = \kappa(G;1,2)=
\#[\mathcal{O}_\EU(G)],\\
|T_G(-1,0)| &= |\bar\kappa(G;-2,-1)| = \kappa(G;2,1) =
\#[\mathcal{O}_\CU(G)];\\
T_G(1,0) &= \bar\kappa(G;0,-1) = |\kappa(G;0,1)| =
\#[\mathcal{O}_\AC(G)],\\
T_G(0,1) &= \bar\kappa(G;-1,0) = |\kappa(G;1,0)| =
\#[\mathcal{O}_\TC(G)];
\end{align*}
\begin{align*}
\kappa(G;2,2) &= \#[\mathcal{O}_\CE(G)].
\end{align*}
Let $[{\mathcal O}(G)]_\CU$ and $[{\mathcal O}(G)]_\EU$ denote the
sets of equivalences classes of ${\mathcal O}(G)$ under the cut
equivalence and Eulerian equivalence relations respectively. Then
\begin{align*}
T_G(1,2) &= \bar\kappa(G;0,1) = \#[{\mathcal O}(G)]_\CU,\\
T_G(2,1) &= \bar\kappa(G;1,0) = \#[{\mathcal O}(G)]_\EU.
\end{align*}
\end{cor}
\begin{proof}
It is completely parallel to the proof of
Corollary~\ref{Integral-Kappa-Special} by modifying the concerned
orientations to proper equivalence classes of those orientations,
except the following two equalities:
\begin{eqnarray*}
\#[{\mathcal O}(G)]_\CU =
\sum_{\rho\in \Rep[{\mathcal O}(G)]} \#[\rho]_\EU,\\
\#[{\mathcal O}(G)]_\EU = \sum_{\rho\in\Rep[{\mathcal O}(G)]}
\#[\rho]_\CU,
\end{eqnarray*}
where $\Rep[{\mathcal O}(G)]$ is a set of distinct representatives
of cut-Eulerian equivalence classes in $[{\mathcal O}(G)]$. The two
equalities follow respectively from the set equations:
\[
[{\mathcal O}(G)]_\CU=\bigsqcup_{\rho\in\Rep[{\mathcal O}(G)]}
\big\{[\sigma]_\CU: \sigma\in[\rho]_\EU\big\},
\]
\[
[{\mathcal O}(G)]_\EU=\bigsqcup_{\rho\in\Rep[{\mathcal O}(G)]}
\big\{[\sigma]_\EU: \sigma\in[\rho]_\CU\big\}.
\]
\end{proof}

\begin{thm}[Integral-Modular Relations]\label{Integral-Kappa-Modular} For each orientation
$\rho\in\mathcal{O}(G)$, let $[\rho]$ denote the cut-Eulerian
equivalence class of $\rho$ in $\mathcal{O}(G)$. Then
\begin{align}
\kappa_{\Bbb Z}(G;x,y) &= \sum_{[\rho]\in[\mathcal{O}(G)]}
\#[\rho]\,\kappa_\rho(G;x,y), \label{Int-Mod-Kappa}\\
\bar\kappa_{\Bbb Z}(G;x,y) &= \sum_{[\rho]\in[\mathcal{O}(G)]}
\#[\rho]\,\bar\kappa_\rho(G;x,y). \label{Int-Mod-Bar-Kappa}
\end{align}
Furthermore, if $\#[\rho]$ is constant for all
$\rho\in\mathcal{O}(G)$, then
\[
\kappa_{\Bbb Z}(G;x,y)=\#[\rho]\,\kappa(G;x,y),
\]
\[
\bar\kappa_{\Bbb Z}(G;x,y)=\#[\rho]\,\bar\kappa(G;x,y).
\]
\end{thm}

Let $y=1$ in Theorem~\ref{Integral-Kappa-Modular}. We obtain a
relation between the integral tension polynomial $\tau_{\Bbb
Z}(G,x)$ and the modular tension polynomial $\tau(G,x)$.

\begin{cor}
Let $[\rho]$ denote the cut equivalence class of $\rho$ in
$\mathcal{O}_\AC(G)$. Then
\[
\tau_{\Bbb Z}(G,x)=\sum_{[\rho]\in[\mathcal{O}_\AC(G)]}
\#[\rho]\,\tau_\rho(G,x).
\]
Furthermore, if $\#[\rho]$ is constant for all
$\rho\in\mathcal{O}_\AC(G)$, then $\tau_{\Bbb
Z}(G,x)=\#[\rho]\,\tau(G,x)$.
\end{cor}

Let $x=1$ in Theorem~\ref{Integral-Kappa-Modular}. We obtain a
relation between the integral flow polynomial $\varphi_{\Bbb
Z}(G,y)$ and the modular flow polynomial $\varphi(G,y)$, which
answers the question asked by Beck and Zaslavsky
\cite{Beck-Zaslavsky}.

\begin{cor}
Let $[\rho]$ denote the Eulerian equivalence class of $\rho$ in
$\mathcal{O}_\TC(G)$. Then
\[
\varphi_{\Bbb Z}(G,y)=\sum_{[\rho]\in[\mathcal{O}_\TC(G)]}
\#[\rho]\,\varphi_\rho(G,y).
\]
Furthermore, if $\#[\rho]$ is constant for all
$\rho\in\mathcal{O}_\TC(G)$, then $\varphi_{\Bbb
Z}(G,y)=\#[\rho]\,\varphi(G,y)$.
\end{cor}

\section{Connection to the Tutte polynomial}

Recall that {\em Whitney's rank generating polynomial}
\cite[p.337]{Bollobas1} of a graph $G$ is
\begin{equation}
R_G(x,y):=\sum_{X\subseteq E} x^{r(G)-r(G|X)} y^{n(G|X)}.
\label{Rank-Generating-Function}
\end{equation}
Shifting each variable in $R_G$ in one unit defines the {\em Tutte
polynomial}
\begin{equation}\label{Tutte-Definition}
T_G(x,y):=R_G(x-1,y-1),
\end{equation}
which can be also defined by recurrence relations; see
\cite[p.339]{Bollobas1}.

\begin{prop}\label{RPQ}
Let $A,B$ be finite abelian groups of orders $|A|=p$, $|B|=q$, and
$\Omega:=\Omega(G,\varepsilon;A,B)$. Then $R_G(x,y)$ has the
following combinatorial interpretations:
\begin{equation}
R_G(p,q) = \sum_{(f,g)\in\Omega\atop \supp g\subseteq \ker f}
2^{|\ker f-\supp g|}, \label{Rank-Generating-Positive}
\end{equation}
\begin{equation}
R_G(-p,-q) = (-1)^{r(G)} \sum_{(f,g)\in\Omega\atop \supp g = \ker f}
(-1)^{|\supp g|}. \label{Rank-Genrating-Negative}
\end{equation}
\end{prop}
\begin{proof}
For subsets $X,Y\subseteq E$, let
$\Omega_{X,Y}=\Omega_{X,Y}(G,\varepsilon;A,B)$ denote the set of
tension-flows $(f,g)\in\Omega$ such that $f|_X=0$, $g|_Y=0$. Then
\[
|\Omega_{X,X^c}| = p^{r(G)-r(G|X)} q^{n(G|X)}.
\]
Thus
\begin{eqnarray*}
R_G(p,q) &=& \sum_{X\subseteq E} |\Omega_{X,X^c}(G,\varepsilon;A,B)| \\
&=& \sum_{X\subseteq E} \sum_{(f,g)\in\Omega\atop \supp g\subseteq
X \subseteq \ker f} 1 \\
&=&  \sum_{(f,g)\in\Omega\atop \supp g\subseteq \ker f} \sum_{\supp
g\subseteq
X\subseteq \ker f\atop} 1 \\
&=& \sum_{(f,g)\in\Omega\atop \supp g\subseteq \ker f} 2^{|\ker
f-\supp g|}.
\end{eqnarray*}

Replace $x$ by $-p$ and $y$ by $-q$ in
(\ref{Rank-Generating-Function}) and note that $|X|=r(G|X)+n(G|X)$
for $X\subseteq E$. We obtain
\begin{eqnarray*}
R_G(-p,-q) &=& (-1)^{r(G)}\sum_{X\subseteq E} (-1)^{|X|} p^{r(G)-r(G|X)} q^{n(G|X)}\\
&=& (-1)^{r(G)} \sum_{X\subseteq E} (-1)^{|X|}
\sum_{(f,g)\in\Omega\atop
\supp g\subseteq X\subseteq \ker f} 1 \\
&=& (-1)^{r(G))} \sum_{(f,g)\in\Omega\atop \supp g\subseteq \ker f}
\sum_{\supp g\subseteq X\subseteq \ker f\atop} (-1)^{|X|}.
\end{eqnarray*}
Using the Binomial Theorem, it is easy to see that
\[
\sum_{\supp g\subseteq X\subseteq \ker f\atop}
(-1)^{|X|}=(-1)^{|\ker f|}\delta_{\supp g,\ker f},
\]
where $\delta_{X,Y}=1$ if $X=Y$ and $\delta_{X,Y}=0$ otherwise.
Therefore we have
\begin{eqnarray*}
R_G(-p,-q) &=& (-1)^{r(G)} \sum_{(f,g)\in\Omega\atop \supp
g\subseteq \ker f}
(-1)^{|\ker f|}\delta_{\supp g,\ker f}\\
 &=& (-1)^{r(G)} \sum_{(f,g)\in\Omega\atop \supp
g = \ker f} (-1)^{|\supp g|}.
\end{eqnarray*}
\end{proof}

Proposition~\ref{RPQ} is due to Reiner \cite{Reiner1}, where
(\ref{Rank-Generating-Positive}) is obtained indirectly for $p,q$ to
be prime powers. 
The following Corollary~\ref{Breuer-Sanyal} is due to Breuer and
Sanyal \cite{Breuer-Sanyal}, obtained by deletion-contraction
method. Its current form is slightly succinct than Breuer and
Sanyal's original statement.

\begin{cor}[Main Result of \cite{Breuer-Sanyal}]\label{Breuer-Sanyal} For positive integers
$p,q$, $T_G(p+1,q+1)$ $(=R_G(p,q))$ equals the number of triples
$(f,g,\rho)$, where $f$ is a ${\Bbb Z}_p$-tension of
$(G,\varepsilon)$ and $g$ is a ${\Bbb Z}_q$-flow of
$(G,\varepsilon)$ such that $\supp g\subseteq\ker f$, and $\rho$ is
a reorientation on the edge subset $\ker f-\supp g$.
\end{cor}
\begin{proof}
Take $A={\Bbb Z}_p$ and $B={\Bbb Z}_q$ in Proposition~\ref{RPQ};
clearly, $|A|=p$ and $|B|=q$. For each tension-flow
$(f,g)\in\Omega(G,\varepsilon;A,B)$ such that $\supp g\subseteq \ker
f$, there are exactly $2^{|\ker f-\supp g|}$ reorientations on the
edge subset $\ker f-\supp g$, since each edge has two choices to be
reoriented. So (\ref{Rank-Generating-Positive}) reduces to the
counting interpretation.
\end{proof}

\begin{thm}\label{RG}
The rank generating polynomial $R_G$ has the following combinatorial
interpretation:
\begin{equation}
R_G(x,y) =\sum_{[\rho]\in[\mathcal{O}(G)]} \bar\kappa_\rho(G;x,y),
\label{Rank-Generating-Interpretation-1}
\end{equation}
\begin{equation}
R_G(-x,-y) = (-1)^{r(G)}\sum_{[\rho]\in[\mathcal{O}(G)]}
(-1)^{|E(C_\rho)|}\,\kappa_\rho(G;x,y).
\label{Rank-Genrating-Negative-Kappa}
\end{equation}
\end{thm}

{\sc First Proof.} Note that a real-valued tension-flow $(f,g)$ of
$(G,\varepsilon)$ is complementary if and only if $\supp g=\ker f$.
Thus
\[
K(G,\varepsilon;{\Bbb Z}_p,{\Bbb Z}_q)
=\{(f,g)\in\Omega(G,\varepsilon;{\Bbb Z}_p,{\Bbb Z}_q):\supp g=\ker
f\}
\]
and (\ref{Rank-Genrating-Negative}) becomes
\[
R_G(-p,-q)=(-1)^{r(G)} \sum_{(f,g)\in K(G,\varepsilon;{\Bbb
Z}_p,{\Bbb Z}_q)} (-1)^{|\supp g|}.
\]
Applying the disjoint decomposition (\ref{Modular-Decom-rho}), we
obtain
\begin{align*}
R_G(-p,-q) &= \sum_{[\rho]\in[\mathcal{O}(G)]} (-1)^{r(G)}
\sum_{(f,g)\in K_\rho(G,\varepsilon;{\Bbb Z}_p,{\Bbb Z}_q)}
(-1)^{|\supp g|}.
\end{align*}
where $K_\rho(G,\varepsilon;{\Bbb Z}_p,{\Bbb
Z}_q):=\Mod_{p,q}\big((p,q)\Delta^{\rho}_\CTF(G,\varepsilon)\cap({\Bbb
Z}\times{\Bbb Z})^E\big)$. For each element $(f,g)$ of
$K_\rho(G,\varepsilon;{\Bbb Z}_p,{\Bbb Z}_q)$, let $(\tilde f,\tilde
g)$ be an element of $(p,q)\Delta^\rho_\CTF(G,\varepsilon)\cap({\Bbb
Z}\times{\Bbb Z})^E$ such that $\Mod_{p,q}(\tilde f,\tilde g)=(f,g)$
by Lemma~\ref{Surjective}. Then $P_{\rho,\varepsilon}(\tilde
f,\tilde g)\in \Delta^+_\CTF(G,\rho)$. Since
$P_{\rho,\varepsilon}(\tilde f+\tilde g) =[\rho,\varepsilon](\tilde
f+\tilde g)>0$, it is clear that $\supp[\rho,\varepsilon]\tilde
g=E(C_\rho)$. Since $\supp[\rho,\varepsilon]\tilde g=\supp\tilde
g=\supp g$, then $\supp g=E(C_\rho)$. Thus
\[
R_G(-p,-q) =(-1)^{r(G)} \sum_{[\rho]\in[\mathcal{O}(G)]}
(-1)^{|E(C_\rho)|} \# K_\rho(G,\varepsilon;{\Bbb Z}_p,{\Bbb Z}_q).
\]
Apply (\ref{Kappa-Z-rho}); we obtain
(\ref{Rank-Genrating-Negative-Kappa}) immediately. The formula
(\ref{Rank-Generating-Interpretation-1}) follows from the
Reciprocity Law (\ref{Local-Reciprocity-Law}) on $\kappa_\rho$ and
$\bar\kappa_\rho$.

{\sc Second Proof (and the Proof of
Theorem~\ref{Rank-Generating-Interpretation} and
Corollary~\ref{Tutte-Interpretation}).} Recall the convolution
formula (\ref{Conv-Formula-Tutte-M}) and the formulas
(\ref{Tau-Tutte}) and (\ref{Phi-Tutte}); we have
\begin{eqnarray*}
T_G(x,y) &=& \sum_{X\subseteq E(G)} T_{G/X}(x,0)\, T_{G|X}(0,y) \\
&=& \sum_{X\subseteq E(G)} (-1)^{r(G/X)+n(G|X)} \tau(G/X,1-x)\,
\varphi(G|X,1-y).
\end{eqnarray*}
Apply the Reciprocity Laws about $\tau,\bar\tau$ and
$\varphi,\bar\varphi$ and their Decomposition Formulas in
\cite{Chen-I,Chen-II}, i.e.,
\[
\bar\tau(G,-x)=(-1)^{r(G)}\tau(G,x), \sp
\bar\varphi(G,-y)=(-1)^{n(G)}\varphi(G,y);
\]
\[
\bar\tau(G,x)=\sum_{\rho\in[{\mathcal O}_\AC(G)]}
\bar\tau_\rho(G,x), \sp \bar\varphi(G,y)=\sum_{\rho\in[{\mathcal
O}_\TC(G)]} \bar\varphi_\rho(G,y).
\]
We then obtain
\begin{eqnarray*}
T_G(x,y) &=& \sum_{X\subseteq E(G)} \bar\tau(G/X,x-1)\,
\bar\varphi(G|X,y-1) \\
&=& \sum_{X\subseteq E(G)} \sum_{[\rho_1]\in[{\mathcal
O}_\AC(G/X)]_\CU\atop [\rho_2]\in[{\mathcal O}_\TC(G|X)]_\EU}
\bar\tau_{\rho_1}(G/X,x-1)\,\bar\varphi_{\rho_2}(G|X,y-1).
\end{eqnarray*}

For each fixed edge subset $X\subseteq E(G)$, let us identify the
orientations on $G|X, G/X$ to orientations on $X, E-X$ respectively.
There is an obvious bijection ${\mathcal O}(G/X)\times{\mathcal
O}(G|X)\rightarrow {\mathcal O}(G)$,
$(\rho_1,\rho_2)\mapsto\rho_1\vee\rho_2$, where $\rho_1\vee\rho_2$
is the orientation on $G$ whose restriction to $X$ is $\rho_2$ and
restriction to $E-X$ is $\rho_1$. Moreover, ${\mathcal O}_\TC(G|X)$
and ${\mathcal O}_\AC(G/X)$ are nonempty if and only if there are no
bridges in $G|X$ and no loops in $G/X$. If $(\rho_1,\rho_2)\in
{\mathcal O}_\AC(G/X)\times{\mathcal O}_\TC(G|X)$, then
\[
C_{\rho_1\vee\rho_2}(G)=X, \sp B_{\rho_1\vee\rho_2}(G)=E(G)-X;
\]
conversely, if $\rho\in{\mathcal O}(G)$, then
$\rho=\rho_1\vee\rho_2$, where $\rho_1=\rho_{/X}$,
$\rho_2=\rho_{|X}$, and $X=C_\rho(G)$; and in both cases
\[
\bar\tau_{\rho_1}(G/X,x)\,\bar\varphi_{\rho_2}(G|X,y)=\bar\kappa_{\rho_1\vee\rho_2}(G;x,y).
\]
We thus have the disjoint unions
\[
{\mathcal O}(G)=\bigsqcup_{X\subseteq E(G)} \big\{\rho_1\vee\rho_2 :
\rho_1\in{\mathcal O}_\AC(G/X),\, \rho_2\in{\mathcal
O}_\TC(G|X)\big\},
\]
\[
[{\mathcal O}(G)]_\CE = \bigsqcup_{X\subseteq E(G)}
\big\{[\rho_1\vee\rho_2] : [\rho_1]\in[{\mathcal O}_\AC(G/X)]_\CU,\,
[\rho_2]\in[{\mathcal O}_\TC(G|X)]_\EU\big\},
\]
where the terms on the right-hand sides may be empty for some $X$.
Therefore
\begin{eqnarray*}
T_G(x,y) &=& \sum_{X\subseteq E(G)\atop {[\rho_1]\in[{\mathcal
O}_\AC(G/X)]_\CU\atop [\rho_2]\in[{\mathcal O}_\TC(G|X)]_\EU}}
\bar\kappa_{\rho_1\vee\rho_2}(G;x-1,y-1)\\
&=& \sum_{[\rho]\in[{\mathcal O}(G)]_\CE}
\bar\kappa_\rho(G;x-1,y-1),
\end{eqnarray*}
which is equivalent to $R_G(x,y)=\bar\kappa(G;x,y)$ by
$R_G(x,y)=T_G(x+1,y+1)$. This is exactly
(\ref{Rank-Generating-Interpretation-1}) and
Theorem~\ref{Rank-Generating-Interpretation}.

Corollary~\ref{Tutte-Interpretation} follows from the fact that a
nonnegative, integer-valued, $(p,q)$-tension-flow $(f,g)$ of a
digraph $(G,\rho)$ is an integer-valued tension-flow such that
$0\leq f(e)\leq p-1$, $0\leq g(e)\leq q-1$ for all $e\in E$.
\hfill{$\Box$}

\section{Example}

Let us consider the graph $G$ in Figure~\ref{8-Form}. Its Tutte
polynomial is given by deletion-contraction as
\[
T(x,y)=y^3+x^2+2xy+2y^2+x+y.
\]
Let $(x_i)$ and $(y_i)$ ($1\leq i\leq 5$) denote respectively
tensions and flows of $G$ with the orientation given in
Figure~\ref{8-Form}. The complementary tension-flows $(x_i,y_i)$
satisfy the following system of linear equations and inequalities:
\[
\begin{array}{r}
x_1-x_2-x_3=0,\\
x_2-x_4=0,\\
x_3-x_5=0,
\end{array}
\begin{array}{r}
y_2-y_3+y_4-y_5=0,\\
y_1+y_2+y_4=0,\\
y_1+y_3+y_5=0,
\end{array}
\begin{array}{r}
x_iy_i=0,\\
x_i+y_i\neq 0,\\
1\leq i\leq 5.
\end{array}
\]

\begin{figure}[h]
\centering
\includegraphics[width=30mm]{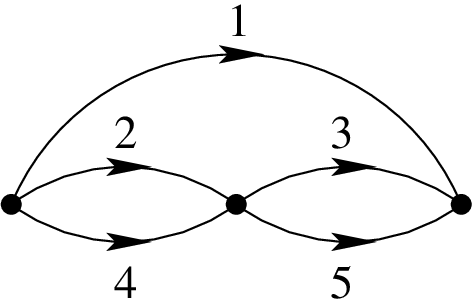}
\caption{The edge-labeled graph $G$.\label{8-Form}}
\end{figure}
When $(x_i)$ is over an abelian group $A$ and $(y_i)$ over another
abelian group $B$, the conditions $x_iy_i=0$ and $x_i+y_i\neq 0$ are
understood as $\supp(x_i)=\ker(y_i)$. Let $|A|=p$ and $|B|=q$; or
set $|x_i|<p$ and $|y_i|<q$ over $\Bbb Z$. By counting the number of
solutions of the above system, we obtain the complementary
polynomial $\kappa$ and the integral complementary polynomial
$\kappa_{\Bbb Z}$ of $G$ as follows:
\[
\kappa(p,q) =(p-1)(p-2)+2(p-1)(q-1)+(q-1)(q-2)^2;
\]
\begin{align*}
\kappa_{\Bbb Z}(p,q) =&3(p-1)(p-2)+8(p-1)(q-1)\\
&+2(q-1)(q-3)(2q-3)+\frac{(q-1)q(2q-1)}{3}.
\end{align*}

There are 8 $(=T(1,1))$ cut-Eulerian equivalence classes of
orientations, 2 ($=T(1,0)$) cut equivalence classes of acyclic
orientations, and 4 ($=T(0,1)$) Eulerian equivalence classes of
totally cyclic orientations. There are 24 ($=T(1,2)$) cut
equivalence classes of orientations (2, 4, 18 in Figures~2, 3, 4,
respectively). There are 14 ($=T(2,1)$) Eulerian equivalence classes
of orientations (6, 4, 4 in Figures~2, 3, 4, respectively). There
are 32 ($=T(2,2)$) total number of orientations. However, there are
essentially only 4 different local complementary polynomials:
\[
\kappa_1(p,q)=\frac{(p-1)(p-2)}{2},
\]
\[
\kappa_2(p,q)=(p-1)(q-1),
\]
\[
\kappa_3(p,q)=-(q-1)^2+\frac{q^2(q-1)}{2} -\frac{(q-1)q(2q-1)}{6},
\]
\[
\kappa_4(p,q)=-(q-1)^2 +\frac{(q-1)q(2q-1)}{6}.
\]
The dual complementary (also rank generating) polynomial
$\bar\kappa$ and the dual integral complementary polynomial
$\bar\kappa_{\Bbb Z}$ are subsequently obtained as follows:
\begin{align*}
\bar\kappa(p,q) &=2[\kappa_1(-p,-q) +\kappa_2(-p,-q)-\kappa_3(-p,-q)
-\kappa_4(-p,-q)] \\
&=(p+1)(p+2)+2(p+1)(q+1)+4(q+1)^2+q^2(q+1)\\
&= q^3+p^2+2pq+5q^2+5p+10q+8 \\
&=T(p+1,q+1);
\end{align*}
\begin{align*}
\bar\kappa_{\Bbb Z}(p,q) &=6\kappa_1(-p,-q)
+8\kappa_2(-p,-q)-8\kappa_3(-p,-q)
-10\kappa_4(-p,-q) \\
&=3(p+1)(p+2)+8(p+1)(q+1)\\
&\hspace{2ex}+18(q+1)^2+4q^2(q+1) +\frac{q(q+1)(2q+1)}{3}.
\end{align*}
Other special values of the polynomials $\kappa$, $\bar\kappa$,
$\kappa_{\Bbb Z}$, and $\bar\kappa_{\Bbb Z}$ are
\[
\kappa(2,1)=T(-1,0)=\#[\mathcal{O}_\CU]=0,
\]
\[
\kappa(1,2)=T(0,-1)=\#[\mathcal{O}_\EU]=0,
\]
\[
\kappa(2,2)=\#[\mathcal{O}_\CE]=0.
\]
The zeroness of these numbers mean that $G$ is {\em not} a cut
(equivalently not bipartite), {\em not} an Eulerian graph, and {\em
not} an edge-disjoint union of a cut and an Eulerian subgraph. The
eight cut-Eulerian equivalence classes of orientations are exhibited
as follows:\vspace{-2ex}

\begin{figure}[h]
\centering \includegraphics[width=58mm]{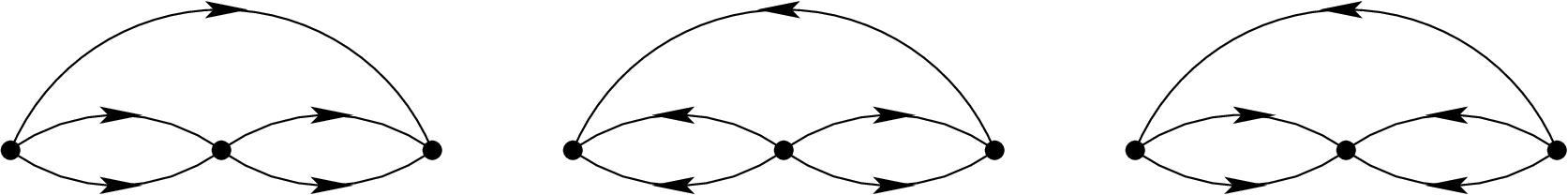}\\
\vspace{2ex}
\includegraphics[width=58mm]{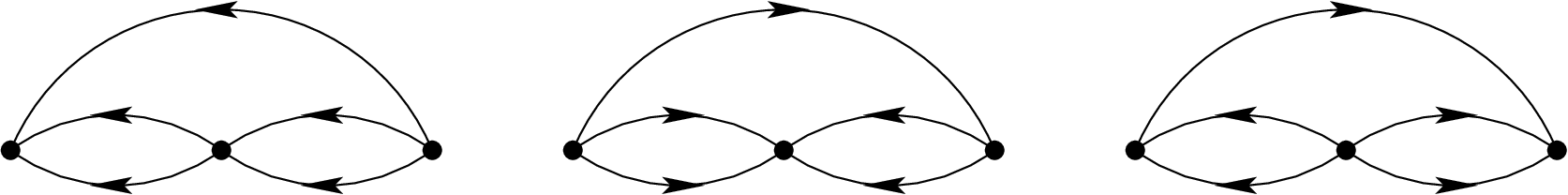} \caption{Two ($=T(1,0)$)
cut equivalence classes of acyclic orientations have the same local
complementary polynomial $\kappa_1$. \label{}}
\end{figure}\vspace{-4ex}

\begin{figure}[h]
\centering
\includegraphics[width=79mm]{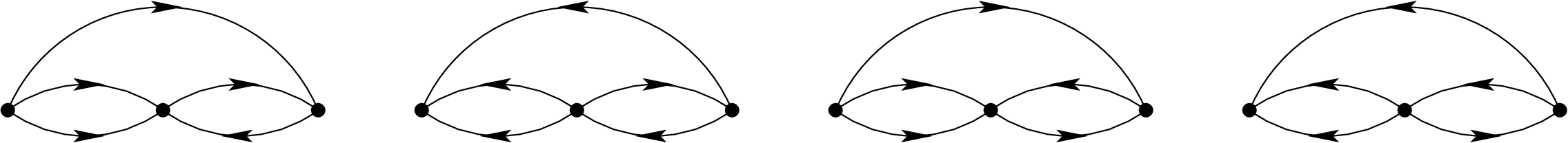}\\
\vspace{2ex}
\includegraphics[width=79mm]{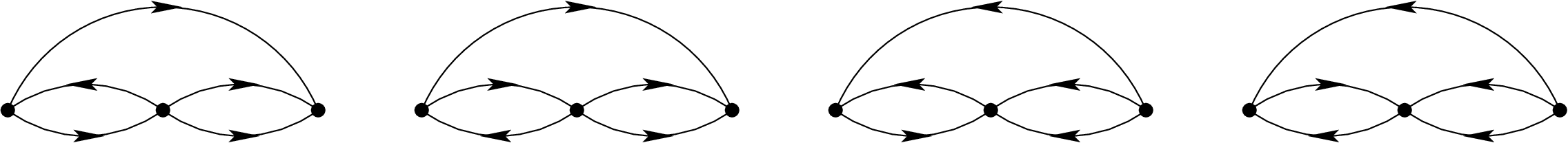} \caption{Two
($=T(1,1)-T(1,0)-T(0,1)$) cut-Eulerian equivalence classes of cyclic
but not totally cyclic orientations have the same local
complementary polynomial $\kappa_2$. \label{Mix}}
\end{figure}\vspace{-5ex}

\begin{figure}[h]
\centering
\includegraphics[width=79mm]{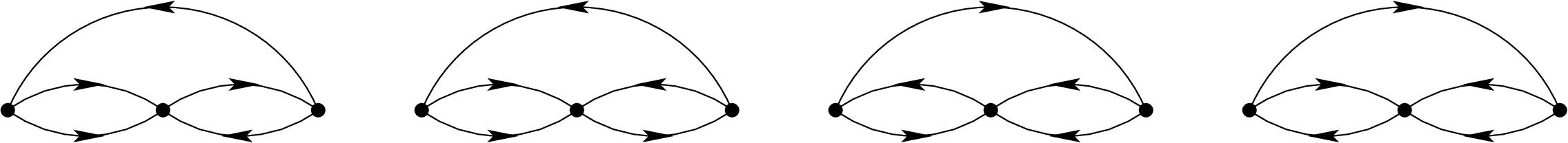}\\
\vspace{2ex}
\includegraphics[width=79mm]{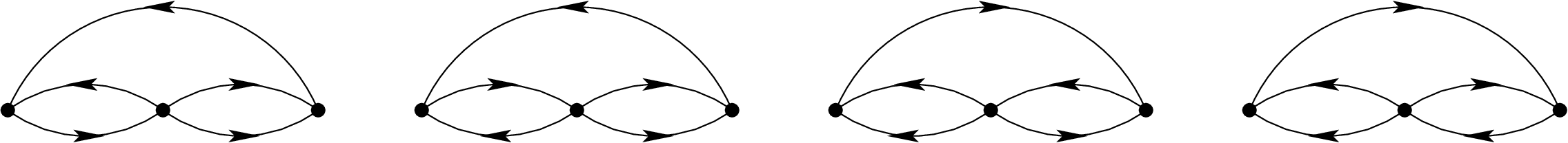}\\
\vspace{2ex}
\includegraphics[width=100mm]{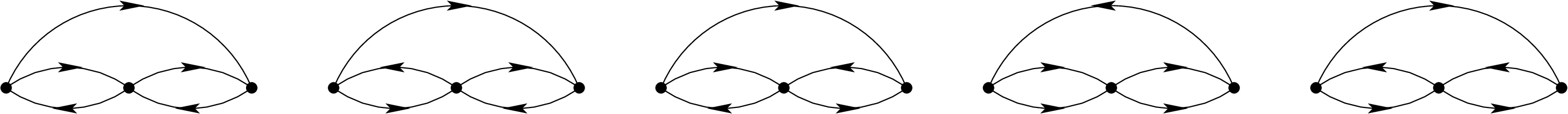}\\
\vspace{2ex}
\includegraphics[width=100mm]{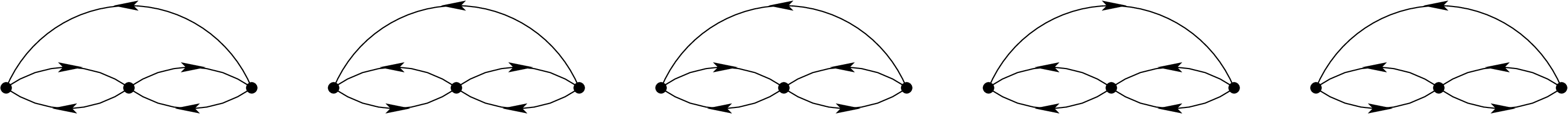}
\caption{Four ($=T(0,1)$) Eulerian equivalence classes of totally
cyclic orientations; the first two and the last two classes have the
same local complementary polynomial $\kappa_3$ and $\kappa_4$,
respectively. \label{}}
\end{figure}

\pagebreak

\begin{center}
{\bf Acknowledgement}
\end{center}

The author thanks Arthur~L.\,B.~Yang for referring the paper
\cite{Kook-Reiner-Stanton} about the Convolution Formula on the
Tutte polynomial and some useful discussions during his visit in May
2007, which leads to the second proof of
Theorem~\ref{Rank-Generating-Interpretation}.

\bibliographystyle{amsalpha}

\end{document}